\documentclass[12pt]{article}

\usepackage{graphicx}
\usepackage{verbatim}
\usepackage{textcomp}
\usepackage{amsmath,amssymb}
\usepackage{esint}
\usepackage{cite}

\newenvironment{proof}%
    {{\sc Proof.}}%
  {{\sc q.e.d.} \\}  

\newcommand{\norm}[1]{\left\Vert#1\right\Vert}
\newcommand{\abs}[1]{\left\vert#1\right\vert}
\newcommand{\R}{\mathbb R}
\newcommand{\Hyp}{\mathbb H}
\newcommand{\tu}{\tilde{u}}
\newcommand{\ra}{\rangle}
\newcommand{\la}{\langle}

\newtheorem{theorem}{Theorem}
\newtheorem{lemma}[theorem]{Lemma}
\newtheorem{proposition}[theorem]{Proposition}
\newtheorem{corollary}[theorem]{Corollary}   
\newtheorem{definition}[theorem]{Definition}

\newtheorem{remark}[theorem]{Remark}

\newtheorem{conjecture}[theorem]{Conjecture}

 \newenvironment{proof1}%
{{\sc Proof of Theorem~\ref{main1}.}}%
{{\sc q.e.d.} \\}

\newenvironment{proof2}%
{{\sc Proof of Corollary~\ref{main2}.}}%
{{\sc q.e.d.} \\}

\newenvironment{proof3}%
{{\sc Proof of Corollary~\ref{main3}.}}%
{{\sc q.e.d.} \\}

\begin{document}

\begin{center}
\Large{A Bochner Formula for Harmonic Maps into Non-Positively Curved Metric Spaces}
\end{center}

\vspace*{0.05in}

\begin{center}
Brian Freidin\\
Brown University\\
{\tt bfreidin@math.brown.edu}
\end{center}

\vspace*{0.1in}

\begin{quote} 
\emph{Abstract.} We study harmonic maps from Riemannian manifolds into arbitrary non-positively curved and CAT(-1) metric spaces. First we discuss the domain variation formula with special emphasis on the error terms. Expanding higher order terms of this and other formulas in terms of curvature, we prove an analogue of the Eels-Sampson Bochner formula in this more general setting. In particular, we show that harmonic maps from spaces of non-negative Ricci curvature into non-positively curved spaces have subharmonic energy density. When the domain is compact the energy density is constant, and if the domain has a point of positive Ricci curvature every harmonic map into an NPC space must be constant.
 \end{quote}

\section{Introduction}

In 1964 Eels and Sampson \cite{eells-sampson} introduced a Bochner identity for harmonic maps between smooth Riemannian manifolds. One of the consequences is that for a harmonic map from a space with a non-negative Ricci tensor to a space of non-positive sectional curvatures, the map is totally geodesic, with constant energy density, and if the domain has any point with positive Ricci curvature, then the map is constant.

The study of harmonic maps has applications in the setting of geometric rigidity. The geometric formulation asks if every map is homotopic to a totally geodesic map. The Bochner formula and other vanishing theorems allow one to deduce a positive answer under further geometric assumptions. In the event that one is looking at equivariant harmonic maps, this sort of rigidity statement implies rigidity statements for the representations of fundamental groups.

Gromov and Schoen \cite{gromov-schoen} initiated the study of harmonic maps into singular spaces, in particular Riemannian simplicial complexes. In \cite{chen} Chen introduced harmonic maps with simplicial complexes as domains as well. In \cite{korevaar-schoen1,korevaar-schoen2,jost1,jost2} the study of harmonic maps is extended to maps into non-positively curved (NPC) metric spaces.

An NPC space $X$ is a geodesic length space (any two points can be connected by a curve whose length realizes their distance) with a comparison principle. For any geodesic triangle $ABC$ in $X$, one can build a Euclidean triangle $\bar{A}\bar{B}\bar{C}\subset\R^2$ with the same side lengths. The NPC criterion is that for any point $D$ on the geodesic $BC$ and the corresponding point $\bar{D}$ on $\bar{C}\bar{D}$ cutting the segment into the same proportions, the distince $d(A,D)$ in $X$ is no more than the distance $\abs{\bar{A}-\bar{D}}$ in $\R^2$. These spaces generalize simply connected manifolds of non-positive sectional curvature.

A CAT(-1) space is simply an NPC space with a stronger comparison principle. Instead of constructing comparison triangles in $\R^2$, construct them in $\Hyp^2$, and the CAT(-1) space has the same comparison inequality.

The introduction of singular spaces into the theory of harmonic maps has allowed for more statements in geometric rigidity. For example, \cite{gromov-schoen} studied harmonic maps into Euclidean buildings to assert the $p$-adic superrigidity of lattices in groups of rank $1$. In \cite{daskal-mese-vdovina} Daskalopoulos, Mese and Vdovina prove the superrigidity of hyperbolic buildings, and in \cite{daskal-mese2} Daskalopoulos and Mese prove similar regularity for maps into more general simplicial complexes. In \cite{daskal-mese1} Daskalopoulos and Mese prove rigidity statements for harmonic maps from $2$-dimensional complexes to general NPC spaces.

In \cite{chen}, Chen used techniques developed in \cite{gromov-schoen} to attack the problem of the Bochner formula for harmonic maps from manifolds of non-negative sectional curvature to non-positively curved simplicial complexes. We expand these methods to derive a Bochner formula for maps into general NPC spaces involving the Ricci curvatures of the domain, and as a result conclude:

\begin{theorem} \label{main1}
For a harmonic map $u:M\to X$ from a Riemannian manifold $M$ to an NPC metric space $X$, $\abs{\nabla u}^2$ satisfies the weak differential inequality\

\[
\frac{1}{2}\Delta\abs{\nabla u}^2 \ge Ric:\pi.
\]
Here $\pi$ is the pull-back metric tensor as defined in \cite{korevaar-schoen1} and $Ric:\pi$ denotes the quantity $g^{ij}g^{k\ell}Ric_{ik}\pi_{j\ell}$. If in addition $X$ is CAT(-1), then
\[
\frac{1}{2}\Delta\abs{\nabla u}^2 \ge Ric:\pi + \abs{\nabla u}^4 - \pi:\pi.
\]
\end{theorem}

A similar result was obtained in \cite{mese}, where Mese studies the Bochner formula on flat domains and concludes $\Delta\abs{\nabla u}^2\ge-2\kappa\abs{\nabla u}^4$ for maps into CAT($\kappa$) spaces. In \cite{mese} Mese also investigates conformal harmonic maps, deriving a curvature bound for the pull back metric in dimension $2$. We will derive an energy bound for conformal harmonic maps from hyperbolic surfaces into CAT(-1) spaces, much as in \cite{wentworth}.

\begin{corollary}\label{main2}
If $u:M\to X$ is a conformal harmonic map from a closed hyperbolic $n$-manifold $M$ into a CAT(-1) space $X$, then $\abs{\nabla u}^2 \le \frac{n}{n-1}$.
\end{corollary}

In \cite{korevaar-schoen2} it is deduced that harmonic maps from flat tori to NPC spaces are totally geodesic. As a byproduct of the methods of this paper, we reproduce this result:

\begin{corollary}\label{main3}
If $M$ is a flat, compact Riemannian manifold and $X$ is an NPC metric space, and $u:M\to X$ is a harmonic map, then $u$ is totally geodesic. That is, constant speed geodesics on $M$ are sent by $u$ to constant speed geodesics in $X$.
\end{corollary}

After first discussing some useful analytic results in \S\ref{Sprelim}, our next order of business is to discuss a domain variation formula in \S\ref{Sdomain}. We derive the formula from \cite{gromov-schoen} via integration by parts, finding one term to have a sign different from what was printed there.

The term of interest derives from the geometry of the domain. In \cite{gromov-schoen}, this and another term were labeled as remainder and shown to have small enough order in terms of other quantities so as not to affect any further results. The sign of these terms is, however, essential to our work.

Also in \S\ref{Sdomain} we derive the target variation formula as well as reproving the Lipschitz continuity of harmonic maps.

In \S\ref{Sbochner} we expand these remainder terms as well as others to ultimately prove a mean value inequality that will lead to the Bochner identity above.

\textbf{Acknowledgements.} The author would like to thank % George Daskalopoulos for suggesting the problem,
Chikako Mese for her interest and comments, as well as Jingyi Chen for his comments and encouragement.

\section{Preliminaries}\label{Sprelim}

Throughout this paper, $M$ denotes a Riemannian manifold of dimension $n$ with Levi-Civita connection $\nabla$ and Riemannian curvature tensor

\[
R(X,Y)Z = \nabla_X\nabla_Y Z - \nabla_Y\nabla_X Z - \nabla_{[X,Y]}Z.
\]
For a local frame $\{e_i\}$ of tangent vectors, define the symbols of the curvature tensor by
\[
R(e_i,e_j)e_k = R_{ijk}{}^\ell e_\ell
\]
and
\[
R_{ijk\ell} = \la R(e_i,e_j)e_k,e_\ell\ra = R_{ijk}{}^m g_{m\ell}.
\]
The symbols of the Ricci curvature now become
\[
R_{ij} = Ric(e_i,e_j) = g^{k\ell}R_{kij\ell} = R_{kij}{}^k.
\]

We will work throughout this paper in normal coordinates around $x_0\in M$, identified with $0\in T_{x_0}M$. We will denote by $d\mu$ and $d\Sigma$ respectively the Riemannian volume and surface measures, and by $dx$ and $dS$ the Euclidean measures. We will also use the convention that $\omega_n$ is the volume of the unit ball in Euclidean $\R^n$. Under all of our conventions, the methods in \cite{g-j-w} show:

\begin{proposition}\label{metric expansion}
On a Riemannian manifold $M$ take normal coordinates $\{x^i\}$ about a point $x_0\in M$. Identifying $x_0$ with $0\in T_{x_0}M$, the metric may be written locally near $x_0$ as

\[
g_{ij} = \delta_{ij} - \frac{1}{3}R_{kij\ell}(0)x^k x^\ell + O(\abs{x}^3).
\]
And, taking determinants, the volume density can be expanded as
\[
\sqrt{g} = 1 - \frac{1}{6}R_{ij}(0)x^i x^j + O(\abs{x}^3).
\]
Here the curvature symbols $R_{kij\ell}$ and $R_{ij}$ are symbols of the curvature tensor at $0=x_0$. Both of the $O(\abs{x}^3)$ terms depend on the geometry of $M$, and in particular are smooth functions of $x$.
\end{proposition}

We will also require the following:

\begin{proposition}\label{quadratic form}
Let $Q$ and $\tilde{Q}$ be quadratic forms on Euclidean space $\R^n$. Write $Q(x) = Q_{ij}x^ix^j$ and $\tilde{Q}(x) = \tilde{Q}_{ij}x^ix^j$ for symmetric matrices $Q_{ij}$ and $\tilde{Q}_{ij}$. Then

\begin{enumerate}
\item\label{quadratic on sphere}
\[
\int_{\partial B_\sigma}Q(x)dS = \omega_n tr(Q)\sigma^{n+1}
\]

\item\label{quadratic in ball}
\[
\int_{B_\sigma}Q(x)dx = \frac{\omega_n}{n+2}tr(Q)\sigma^{n+2}
\]

\item\label{product quadratic on sphere}
\[
\int_{\partial B_\sigma}Q(x)\tilde{Q}(x)dS = \frac{\omega_n}{n+2}\left(2Q:\tilde{Q}+tr(Q)tr(\tilde{Q})\right)\sigma^{n+3}.
\]
\end{enumerate}
Here the notation $Q:\tilde{Q}$ denotes the trace of the product of the matrices, $Q:\tilde{Q} = Q_{ij}\tilde{Q}_{ji}$.
\end{proposition}

\begin{proof}
For \ref{quadratic on sphere}, see

\[
\int_{\partial B_\sigma}Q(x)dS = \int_{\partial B_\sigma}Q_{ij}x^ix^j.
\]
For $i\neq j$, this integral vanishes since the integrand is an odd function of $x^i$. So
\begin{eqnarray*}
\int_{\partial B_\sigma}Q(x)dS & = & Q_{ii} \int_{\partial B_\sigma}(x^i)^2dS\\
 & = & \frac{1}{n}tr(Q)\int_{\partial B_\sigma}\sigma^2dS\\
 & = & \omega_n tr(Q) \sigma^{n+1}.
\end{eqnarray*}

\ref{quadratic in ball} follows from integrating \ref{quadratic on sphere}. For \ref{product quadratic on sphere} compute

\[
\int_{\partial B_\sigma}Q(x)\tilde{Q}(x)dS = \int_{\partial B_\sigma}Q_{ij}x^ix^j \tilde{Q}_{k\ell}x^kx^\ell.
\]

Define a vector field $V = \sigma Q_{ij}x^ix^j\tilde{Q}_{k\ell}x^ke_\ell$ and note that on $\partial B_\sigma$ the unit outward normal is $\nu = \frac{1}{\sigma}x^me_m$. Now the divergence theorem yields the result.
%\begin{eqnarray*}
%\int_{\partial B_\sigma}Q(x)\tilde{Q}(x)dS & = & \int_{\partial B_\sigma}V\cdot\nu dS\\
% & = & \int_{B_\sigma}div(V)dx\\
% & = & \sigma\int_{B_\sigma}\frac{\partial}{\partial x^\ell}\left(Q_{ij}x^ix^j\tilde{Q}_{k\ell}x^k\right)dx\\
% & = & \sigma\left[ R_{kk}\int_{B_\sigma}Q_{ij}x^ix^jdx + 2\int_{B_\sigma} Q_{ij}\tilde{Q}_{jk}x^ix^kdx\right]\\
% & = & \frac{\omega_n}{n+2}\left(2Q_{ij}\tilde{Q}_{ji} + Q_{ii}\tilde{Q}_{jj}\right)\sigma^{n+3}
%\end{eqnarray*}
\end{proof}

Two special cases of part \ref{quadratic in ball} are as follows: If we take the Euclidean metric $g\big|_{x_0}$ in $T_{x_0}M$, and let $Q(x) = Ric_{x_0}(x,x)$, then $Q_{ij} = Ric_{ij}(0)$ are the symbols of the Ricci tensor, whose trace $tr(Ric) = S$ is the scalar curvature, so

\begin{equation}\label{integrate ricci}
\int_{B_\sigma(0)}Ric_{x_0}(x,x)dx = \frac{\omega_n}{n+2}S(x_0)\sigma^{n+2}.
\end{equation}

If we fix a vector $v\in T_{x_0}M$ and take $Q(x) = \la R_{x_0}(x,v)v,x\ra$, then $Q_{ij} = v^kv^\ell R_{ik\ell j}(x_0)$ and $tr(Q) = v^kv^\ell R_{ik\ell i}(x_0) = v^kv^\ell Ric_{k\ell}(x_0) = Ric_{x_0}(v,v)$, so

\begin{equation}\label{integrate sectional}
\int_{B_\sigma(0)}\la R_{x_0}(v,x)x,v\ra dx = \frac{\omega_n}{n+2}Ric_{x_0}(v,v)\sigma^{n+2}.
\end{equation}

We also now have the tools to phrase a more precise asymptotic version of the Bishop-Gromov comparison theorem:

\begin{proposition}\label{bishop-gromov}
Let $M$ be a Riemannian manifold of dimension $n$. Then for $x_0\in M$ and $\sigma$ small,

\[
\abs{\partial B_\sigma(x_0)} = \left(\frac{n}{\sigma} - \frac{1}{3(n+2)}S(x_0)\sigma + O(\sigma^2)\right)\abs{B_\sigma(x_0)}.
\]
\end{proposition}

\begin{proof}
The ball of radius $\sigma$ in normal coordinates coincides with the Euclidean ball of the same radius, and we have the Taylor expansion of the volume density from Proposition~\ref{metric expansion}:

\[
\sqrt{g} = 1 - \frac{1}{6}R_{ij}(0)x^ix^j + O(\abs{x}^3).
\]
Integrate to find the volume of the ball, using part \ref{quadratic in ball} from Proposition~\ref{quadratic form}:
\begin{eqnarray*}
\abs{B_\sigma} & = & \int_{B_\sigma}\sqrt{g}dx\\
 & = & \int_{B_\sigma}\left(1-\frac{1}{6}R_{ij}x^ix^j + O(\abs{x}^3)\right)dx\\
 & = & \omega_n\sigma^n - \frac{\omega_n}{6(n+2)}S(0)\sigma^{n+2} + O(\sigma^{n+3})\\
\frac{1}{\abs{B_\sigma}} & = & \frac{1}{\omega_n\sigma^n}\left( 1 + \frac{1}{6(n+2)}S(0)\sigma^2 + O(\sigma^3)\right).
\end{eqnarray*}

Now the area of the sphere can be found by differentiating the above, or by integrating the volume density on the sphere:

\begin{eqnarray*}
\abs{\partial B_\sigma} & = & n\omega_n\sigma^{n-1} - \frac{\omega_n}{6}S(0)\sigma^{n+1} + O(\sigma^{n+2})\\
\frac{\abs{\partial B_\sigma}}{\abs{B_\sigma}} & = & \left( \frac{n}{\sigma} - \frac{1}{6}S(0)\sigma + O(\sigma^2)\right) \left( 1 + \frac{1}{6(n+2)}S(0)\sigma^2 + O(\sigma^3)\right)\\
% & = & \frac{n}{\sigma} + \left(\frac{n}{6(n+2)} - \frac{1}{6}\right)S(0)\sigma + O(\sigma^2)\\
 & = & \frac{n}{\sigma} - \frac{1}{3(n+2)}S(0)\sigma + O(\sigma^2).
\end{eqnarray*}
\end{proof}

Throughout much of the paper, we will be expanding non-smooth terms, and we will need to understand how small the error terms are. The following definition captures the features of the error terms we will accumulate that are essential to showing they are small enough.

\begin{definition}
A term $o(\sigma^k)$ denotes a function on $M$ depending on the parameter $\sigma$ with the property that for $\sigma$ less than some $\sigma_0$, $\norm{\frac{o(\sigma^k)}{\sigma^k}}_{L^\infty}$ is finite and bounded independent of $\sigma$ (i.e. an $o(\sigma^k)$ is $O(\sigma^k)$), and for almost every $x\in M$, $\frac{o(\sigma^k)}{\sigma^k}\to 0$ as $\sigma\to 0$.
\end{definition}

\begin{remark}
By the dominated convergence theorem, a term that is $o(\sigma^k)$ has, for any compact $K\subseteq M$,

\[
\norm{\frac{o(\sigma^k)}{\sigma^k}}_{L^1(K)}\to 0.
\]
Note also that a smooth $O(\sigma^k)$ term is, in particular, $o(\sigma^{k-1})$.
\end{remark}

We will need to do some arithmetic with $o(\sigma^k)$ terms, as well as integrate such a function over a range of $\sigma$.

\begin{lemma}\label{integrate-little o}
For a function $o(\sigma^k)$, the function
\[
\int_0^\sigma o(\tau^k)d\tau
\]
defined by integrating $o(\sigma^k)$ pointwise on $M$ is $o(\sigma^{k+1})$. Additionally,

\begin{itemize}
\item $o(\sigma^k)o(\sigma^\ell) = o(\sigma^{k+\ell})$
\item $e^{o(\sigma^k)} = 1 + o(\sigma^k)$
\item $\frac{1}{1+o(\sigma^k)} = 1 + o(\sigma^k)$.
\end{itemize}
\end{lemma}

\begin{proof}
We will prove the integral statement. The other three follow similarly, the latter two with Taylor expansions. For $\sigma$ sufficiently small we have $\abs{\frac{o(\sigma^k)}{\sigma^k}}\le C$ independent of $\sigma$. Hence

\begin{eqnarray*}
\abs{\frac{1}{\sigma^{k+1}}\int_0^\sigma o(\tau^k)(x)d\tau} & \le & \frac{1}{\sigma^{k+1}} \int_0^\sigma C\tau^k d\tau\\
 & = & \frac{C}{k+1}.
\end{eqnarray*}

And for $\epsilon>0$ and almost any $x\in M$, we can find $\sigma_0>0$ so that $\sigma<\sigma_0$ implies $\frac{1}{\sigma^k}o(\sigma^k)(x) < \epsilon$. So for $\sigma < \sigma_0$ we compute

\begin{eqnarray*}
\frac{1}{\sigma^{k+1}}\int_0^\sigma o(\tau^k)(x) d\tau & < & \frac{1}{\sigma^{k+1}}\int_0^\sigma \epsilon\tau^kd\tau\\
 & = & \frac{\epsilon}{k+1}.
\end{eqnarray*}

Since $\frac{1}{\sigma^{k+1}}\int_0^\sigma o(\tau^k)d\tau$ is bounded in $L^\infty$ independent of $\sigma$ and tends to $0$ almost everywhere, $\int_0^\sigma o(\tau^k)d\tau = o(\sigma^{k+1})$.
\end{proof}

The way in which most of the $o(\sigma^k)$ error terms will arise is by pulling non-smooth functions out of integrals. The following propositioni allows us to do this.

\begin{proposition}\label{lebesgue-little o}
For bounded functions $f$ and $\phi$ on a domain $\Omega\subset M$, integrating in normal coordinates around $x_0\in\Omega$ we have for $\sigma < d(x_0,\partial \Omega)$

\[
\int_{B_\sigma(x_0)}f(x)\phi(x)dx = \phi(x_0)\int_{B_\sigma(x_0)}f(x)dx + o(\sigma^n)\sup_{B_\sigma(x_0)}\abs{f}.
\]
\end{proposition}

\begin{proof}
First observe
\[
\abs{\int_{B_\sigma(x_0)}f(x)\phi(x)dx - \int_{B_\sigma(x_0)}f(x)\phi(x_0)dx} \le \sup_{B_\sigma(x_0)}\abs{f}\int_{B_\sigma(x_0)}\abs{\phi(x)-\phi(x_0)}dx.
\]

Define $\phi^\sigma(x_0) = \int_{B_\sigma(x_0)}\abs{\phi(x)-\phi(x_0)}dx$. At a Lebesgue point $x_0$ for $\phi$, we have $\frac{1}{\sigma^n}\phi^\sigma(x_0)\to 0$. This is the almost everywhere convergence condition.

Since $\phi\in L^\infty$, there is some $K>0$ with $\abs{\phi}\le K$, so $\phi^\sigma \le 2K\omega_n\sigma^n$. This gives the $L^\infty$ bound. Hence $\phi^\sigma = o(\sigma^n)$, so we can write

\[
\int_{B_\sigma}f(x)\phi(x)dx = \int_{B_\sigma}f(x)\phi(0)dx + o(\sigma^n)\sup_{B_\sigma}\abs{f}.
\]
\end{proof}

\section{Domain and Target Variation Formulas and Lipschitz Continuity}\label{Sdomain}

We will now derive the formula from \cite{gromov-schoen}, taking special care with terms involving derivatives of metric information, as these will be essential to the rest of our work.

\begin{proposition}
For a Riemannian manifold $M$, an NPC metric space $X$, and a harmonic map $u:M\to X$, we have for almost every $x_0\in M$ (writing $B_\sigma$ for $B_\sigma(x_0)$),

\begin{align}
(2-n)\int_{B_\sigma} \abs{\nabla u}^2 d\mu & + \sigma\int_{\partial B_\sigma} \left[ \abs{\nabla u}^2 - 2 \abs{\frac{\partial u}{\partial r}}^2\right] d\Sigma\nonumber\\
& = \int_{B_\sigma} \left[ x^k\frac{\partial g^{ij}}{\partial x^k} \frac{\partial u}{\partial x^i}\cdot\frac{\partial u}{\partial x^j}\sqrt{g} + \abs{\nabla u}^2 x^k\frac{\partial\sqrt{g}}{\partial x^k}\right]dx.\label{monotonicity-gs}
\end{align}
\end{proposition}

\begin{remark}
Note that this formula corrects a sign error appearing in \cite{gromov-schoen}. The change to one term on the right hand side of the above formula, while essential to the present work, does not affect the remainder of the arguments in \cite{gromov-schoen} as they only require these terms to be $O(\sigma^2)\int_{B_\sigma}\abs{\nabla u}^2d\mu$.
\end{remark}

\begin{proof}
These computations will be done in normal coordinates around $x_0 = 0$. Consider a family of maps $u_t:M\to X$ as follows: Let $\eta$ be a smooth test function supported near $0$, and let $F_t(x) = (1+t\eta(x))x$. Define $u_t(x) = u\circ F_t(x)$.

\cite{korevaar-schoen1} describes how to relate the pull-back tensor of $u_t$ with that of $u$ by the formula
\[
(\pi_{u_t})_{ij}(x) = \frac{\partial F_t^\alpha}{\partial x^i}(x)\frac{\partial F_t^\beta}{\partial x^j}(x)(\pi_u)_{\alpha\beta}(F_t(x)).
\]
This formula generalizes the chain rule for smooth functions. Now we consider the energy of the map $u_t$:

\begin{eqnarray*}
E(u_t) & = & \int {\nabla u_t}^2 d\mu\\
 & = & \int g^{ij}(x)(\pi_{u_t})_{ij}(x)d\mu\\
 & = & \int g^{ij}(x)(\pi_u)_{\alpha\beta}(F_t(x)) \left[(1+t\eta(x))\delta_i^\alpha + tx^\alpha\frac{\partial\eta}{\partial x^i}(x)\right]\\
 & & \times\left[(1+t\eta(x))\delta_j^\beta + tx^\beta\frac{\partial\eta}{\partial x^j}(x)\right]d\mu\\
 & = & \int g^{ij}(x)\left[(1+2t\eta(x))(\pi_u)_{ij}(F_t(x)) + 2tx^k\frac{\partial\eta}{\partial x^i}(x)(\pi_u)_{jk}(F_t(x))\right] d\mu + O(t^2)\\
 & = & O(t^2) + \int g^{ij}(F_t^{-1}(x)) \Bigg[ (1+2t\eta(F_t^{-1}(x))) (\pi_u)_{ij}(x)\\
 & & + 2t(F_t^{-1})^k(x) \frac{\partial\eta}{\partial x^i}(F_t^{-1}(x)) (\pi_u)_{jk}(x) \Bigg] \abs{det(dF_t^{-1}(x))}\sqrt{g(F_t^{-1}(x))}dx.
\end{eqnarray*}

This last line comes from the change of variables formula, which is valid even in the presence of the $L^1$ tensor $\pi$. First compute

\begin{eqnarray*}
\frac{\partial F_t^i}{\partial x^j} & = & (1+t\eta(x))\delta^i_j + t\frac{\partial \eta}{\partial x^j}x^i\\
det(dF_t) & = & 1 + nt\eta(x) + tx^k\frac{\partial\eta}{\partial x^k} + O(t^2)\\
det(dF_t^{-1}) & = & 1 - nt\eta - tx^k\frac{\partial\eta}{\partial x^k} + O(t^2).
\end{eqnarray*}

We will need the following computation, for smooth functions $f$:

\[
\frac{\partial}{\partial t} f\circ F_t^{-1}(x) = \frac{\partial f}{\partial x^k}(F_t^{-1}(x))\frac{\partial (F_t^{-1})^k}{\partial t}(x).
\]

Since $F_t(x) = (1+t\eta(x))x$, there is some function $\rho$ so that $F_t^{-1}(x) = \rho(t,x)x$. Then $x = F_t(F_t^{-1}(x)) = (1+t\eta(\rho(t,x)x))\rho(t,x)x = x$. Hence $(1+t\eta(\rho(t,x)x))\rho(t,x) = 1$ and so taking $\frac{\partial}{\partial t}$ we have

\[
0 = \eta(\rho(t,x)x)\rho(t,x) + (1+t\eta(\rho(t,x)x))\frac{\partial\rho}{\partial t}.
\]

At $t=0$, $\rho(0,x) = 1$ since $F_0$ is the identity, so $\frac{\partial\rho}{\partial t}\big|_{t=0} = -\eta(x)$. Now $(F_t^{-1})^k(x) = \rho(t,x)x^k$, so we have the following formula for smooth functions $f$ (which we will apply for $f = \sqrt{g}$ and $f = g^{ij}$):

\[
\frac{\partial}{\partial t}\big|_{t=0} f\circ F_t^{-1}(x) = -\eta(x)x^k\frac{\partial f}{\partial x^k}(x).
\]

Now we can take the derivative of $E(u_t)$ at $t=0$ and pass the derivative under the integral, writing $\pi$ for $\pi_u$. Since $u$ is harmonic, the derivative is $0$:

\begin{eqnarray*}
0 & = & \frac{d}{dt}\big|_{t=0} E(u_t)\\
 & = & -\int \eta x^k\frac{\partial g^{ij}}{\partial x^k}\pi_{ij}d\mu + \int g^{ij} \left[ 2\eta\pi_{ij} + 2x^k \frac{\partial\eta}{\partial x^i}\pi_{jk}\right] d\mu\\
 & & - \int g^{ij}\pi_{ij} \left(n\eta + x^k\frac{\partial\eta}{\partial x^k}\right) d\mu - \int g^{ij}\pi_{ij}\eta x^k\frac{\partial\sqrt{g}}{\partial x^k}dx.
\end{eqnarray*}
In other words,
%\begin{equation}\label{pre-monotonicity}
%0 = \int g^{ij}(x)\left[\frac{d}{dt}|_{t=0}\left(\pi_{ij}(F_t(x))\right) + 2\eta\pi_{ij} + 2x^k\frac{\partial\eta}{\partial x^i}(x)\pi_{jk}(F_t(x))\right]d\mu.
%\end{equation}

%By the chain rule compute

%\[
%\frac{d}{dt}(f(F_t(x))) = \frac{\partial f}{\partial x^k}\frac{\partial F_t^k}{\partial t} = \eta x^k\frac{\partial f}{\partial x^k}
%\]

%So the first term in \eqref{pre-monotonicity} is

%\[
%\int \eta g^{ij} x^k\frac{\partial}{\partial x^k}(\pi_{ij})\sqrt{g}dx
%\]

%Integrating by parts, this becomes

%\begin{eqnarray*}
%-\int \frac{\partial}{\partial x^k}\left(\eta x^k g^{ij} \sqrt{g}\right)\pi_{ij}dx %& = & -\int \left(n\eta + x^k\frac{\partial\eta}{\partial x^k}\right)|\nabla u|^2d\mu\\
% & & - \int \eta\left[\pi_{ij}x^k\frac{\partial g^{ij}}{\partial x^k}\sqrt{g} + |\nabla u|^2 x^k\frac{\partial \sqrt{g}}{\partial x^k}\right]dx
%\end{eqnarray*}

%Combining these terms into \eqref{pre-monotonicity} and moving terms involving derivatives of the metric to the other side, see

\begin{eqnarray*}
(2-n)\int \eta\abs{\nabla u}^2d\mu & + & \int \left[2g^{ij}x^k\frac{\partial\eta}{\partial x^i}\pi_{jk} - x^k\frac{\partial\eta}{\partial x^k}\abs{\nabla u}^2\right]d\mu\\
 & = & \int \eta\left[\pi_{ij}x^k\frac{\partial g^{ij}}{\partial x^k}\sqrt{g} + \abs{\nabla u}^2 x^k\frac{\partial \sqrt{g}}{\partial x^k}\right]dx.
\end{eqnarray*}

Let now $\eta$ approach the characteristic function of a ball $B_\sigma$ about $0$. The partial derivatives $\frac{\partial\eta}{\partial x^k}$ approach $-\frac{x^k}{\abs{x}}$ times a $\delta$-distribution on the sphere $\partial B_\sigma$. So the above formula becomes

\begin{eqnarray*}
(2-n)\int_{B_\sigma}\abs{\nabla u}^2 d\mu & + & \int_{\partial B_\sigma} \left[ \frac{x^kx^k}{\abs{x}}\abs{\nabla u}^2 - 2g^{ij}\frac{x^kx^i}{\abs{x}}\pi_{jk}\right]d\Sigma\\
 & = & \int_{B_\sigma}\left[\pi_{ij}x^k\frac{\partial g^{ij}}{\partial x^k}\sqrt{g} + \abs{\nabla u}^2x^k\frac{\partial\sqrt{g}}{\partial x^k}\right]dx.
\end{eqnarray*}

Now note in the second integral on the left that $x^kx^k =\abs{x}^2 = \sigma^2$. From \cite{g-j-w} equation $(2.8)$ says $x^ig_{ij}=x^j$, so $g^{ij}x^ix^k\pi_{jk} = x^jx^k\pi_{jk}$. And $x^jx^k\pi_{jk} = \pi(x^j\partial_j,x^k\partial_k) = \pi\left( r\frac{\partial}{\partial r},r\frac{\partial}{\partial r}\right) = r^2\abs{\frac{\partial u}{\partial r}}^2$. Hence

\begin{eqnarray*}
(2-n)\int_{B_\sigma}\abs{\nabla u}^2 d\mu & + & \sigma\int_{\partial B_\sigma}\left[\abs{\nabla u}^2 - 2\abs{\frac{\partial u}{\partial r}}^2\right]d\Sigma\\
 & = & \int_{B_\sigma}\left[\pi_{ij}x^k\frac{\partial g^{ij}}{\partial x^k}\sqrt{g} + \abs{\nabla u}^2x^k\frac{\partial\sqrt{g}}{\partial x^k}\right]dx.
\end{eqnarray*}
\end{proof}

For the rest of our work, we will need a more precise expansion of the terms on the right hand side of this formula.

\begin{lemma} \label{monotonicity1}
Let $M$ be a Riemannian manifold, $X$ an NPC metric space, and $u:M\to X$ a harmonic map. For almost every $x_0\in M$,

\begin{align*}
(2-n) & \int_{B_\sigma(x_0)} \abs{\nabla u}^2 d\mu + \sigma\int_{\partial B_\sigma(x_0)} \left[ \abs{\nabla u}^2 - 2 \abs{\frac{\partial u}{\partial r}}^2\right] d\Sigma\\
 & = \frac{\omega_n\left(2Ric:\pi(x_0) - S(x_0)\abs{\nabla u}^2(x_0)\right)}{3(n+2)}\sigma^{n+2} + o(\sigma^{n+2}).
\end{align*}
Here $Ric:\pi$ denotes the quantity $g^{ik}g^{j\ell}\pi_{ij}R_{k\ell}$, which at the center of normal coordinates becomes $\pi_{ij}Ric_{ij}$. If $\abs{\nabla u}^2(x_0)\neq 0$, then
\begin{align*}
(2-n) & \int_{B_\sigma(x_0)} \abs{\nabla u}^2 d\mu + \sigma\int_{\partial B_\sigma(x_0)} \left[ \abs{\nabla u}^2 - 2 \abs{\frac{\partial u}{\partial r}}^2\right] d\Sigma\\
 & = \left(\frac{1}{3(n+2)}\left(\frac{2Ric:\pi(x_0)}{\abs{\nabla u}^2(x_0)} - S(x_0)\right)\sigma^2 + o(\sigma^2)\right)\int_{B_\sigma(x_0)}\abs{\nabla u}^2d\mu.
\end{align*}
\end{lemma}

\begin{proof}
This result follows from taking a closer look at the terms on the right hand side in \eqref{monotonicity-gs}. We will first expand the terms that come from the domain, those that do not involve the tensor $\pi$. Fix $x_0\in M$, a Lebesgue point for all the $\pi_{ij}$ and $\abs{\nabla u}^2$, and identify it with $0\in T_{x_0}M$ via exponential coordinates. Recall the Taylor expansion for the metric and the volume density from Proposition~\ref{metric expansion} and compute:

\begin{eqnarray*}
%g_{ij} & = & \delta_{ij} - \frac{1}{3}R_{i\ell mj}x^\ell x^m + O(\abs{x}^3)\\
%\sqrt{g} & = & 1 - \frac{1}{6}R_{ij}x^ix^j + O(\abs{x}^3)\\
%g^{ij} & = & \delta^{ij}+O(\abs{x}^2)\\
%x^k\frac{\partial}{\partial x^k}g_{ij} & = & -\frac{2}{3}R_{i\ell mj}x^\ell x^m + O(\abs{x}^3)\\
% & = & -\frac{2}{3}\langle R_0(\partial_i, x)x,\partial_j\rangle + O(\abs{x}^3)\\
x^k\frac{\partial}{\partial x^k}g^{ij} %& = & -g^{ip}x^k\frac{\partial g_{pq}}{\partial x^k}g^{qj}\\
 & = & \frac{2}{3} \langle R_0(\partial_i,x)x,\partial_j\rangle + O(\abs{x}^3)\\
%x^k\frac{\partial g^{ij}}{\partial x^k}\sqrt{g} & = & \frac{2}{3}\langle R_0(\partial_i,x)x,\partial_j\rangle + O(\abs{x}^3)\\
x^k\frac{\partial\sqrt{g}}{\partial x^k} %& = & -\frac{1}{3}R_{ij}x^ix^j + O(\abs{x}^3)\\
 & = & -\frac{1}{3}Ric_0(x,x) + O(\abs{x}^3).
\end{eqnarray*}

Combining all the $O(\abs{x}^3)$ terms and recalling that $\abs{\pi_{ij}}\le\abs{\nabla u}^2$, we may now rewrite the right hand side of \eqref{monotonicity-gs} as

\[
\frac{2}{3}\int_{B_\sigma}\pi_{ij}\langle R_0(\partial_i,x)x,\partial_j \rangle dx - \frac{1}{3}\int_{B_\sigma}\abs{\nabla u}^2 Ric_0(x,x) dx + O(\sigma^3)\int_{B_\sigma}\abs{\nabla u}^2dx.
\]

Since $\abs{\nabla u}^2$ is in $L^\infty_{loc}$ (because $u$ is locally Lipschitz), it is in $L^\infty(\Omega)$ for some $\Omega\subset M$ compact, and all the quadratic curvature terms are smooth and $O(\abs{x}^2)$, so Proposition~\ref{lebesgue-little o} lets us write the above as
\[
\frac{2}{3}\pi_{ij}(0)\int_{B_\sigma}\langle R_0(\partial_i,x)x,\partial_j\rangle dx - \frac{1}{3}\abs{\nabla u}^2(0)\int_{B_\sigma}Ric_0(x,x)dx + o(\sigma^{n+2}).
\]
Note that $\int_{B_\sigma}\abs{\nabla u}^2d\mu$ is $O(\sigma^n)$, and so $O(\sigma^3)\int_{B_\sigma}\abs{\nabla u}^2dx$ is $O(\sigma^{n+3})$, which is in particular $o(\sigma^{n+2})$, so it has been absorbed into that term.

Now choose an orthonormal change of basis to diagonalize $\pi(0)$, so that $\pi_{ij}(0) = \lambda_i\delta_{ij}$. Now the right hand side of \eqref{monotonicity-gs} can be further rewritten
\[
\frac{2}{3}\lambda_i\int_{B_\sigma}\langle R_0(\partial_i,x)x,\partial_i\rangle dx - \frac{1}{3}\abs{\nabla u}^2(0)\int_{B_\sigma}Ric_0(x,x)dx + o(\sigma^{n+2}).
\]
Equations \eqref{integrate ricci} and \eqref{integrate sectional} after Proposition~\ref{quadratic form} let us integrate these curvature terms. The above becomes
\[
\frac{\omega_n}{3(n+2)}\left(2Ric:\pi(0) - S(0)\abs{\nabla u}^2(0)\right)\sigma^{n+2}+o(\sigma^{n+2}).
\]
Hence \eqref{monotonicity-gs} can be written

\begin{align*}
(2-n) & \int_{B_\sigma(x_0)} \abs{\nabla u}^2 d\mu + \sigma\int_{\partial B_\sigma(x_0)} \left[ \abs{\nabla u}^2 - 2 \abs{\frac{\partial u}{\partial r}}^2\right] d\Sigma\\
 & = \frac{\omega_n}{3(n+2)}\sigma^{n+2}\left(2\frac{Ric:\pi(x_0)}{\abs{\nabla u}^2(x_0)}-S(x_0)\right)\abs{\nabla u}^2(x_0) + o(\sigma^{n+2}).
\end{align*}

Since $x_0$ is a Lebesgue point for $\abs{\nabla u}^2$, we have

\begin{eqnarray*}
\int_{B_\sigma(x_0)}\abs{\nabla u}^2d\mu & = & \abs{\nabla u}^2(x_0)Vol(B_\sigma) + o(\sigma^n)\\
 & = & \abs{\nabla u}^2(x_0)\omega_n\sigma^n + o(\sigma^n).
\end{eqnarray*}
The $o(\sigma^n)$ error term in the first line is integrable, and small, by the same reasoning as in Proposition~\ref{lebesgue-little o}. So as long as $\abs{\nabla u}^2(x_0)\neq 0$, we can write
\[
\left(\int_{B_\sigma(x_0)}\abs{\nabla u}^2d\mu\right)^{-1} = \frac{1}{\omega_n\sigma^n\abs{\nabla u}^2(x_0)}(1+o(1)).
\]

So around points with $\abs{\nabla u}^2\neq 0$ \eqref{monotonicity-gs} can be further written

\begin{align*}
(2-n) & \int_{B_\sigma(x_0)} \abs{\nabla u}^2 d\mu + \sigma\int_{\partial B_\sigma(x_0)} \left[ \abs{\nabla u}^2 - 2 \abs{\frac{\partial u}{\partial r}}^2\right] d\Sigma\\
 & = \left(\frac{1}{3(n+2)}\left(\frac{2Ric:\pi(x_0)}{\abs{\nabla u}^2(x_0)} - S(x_0)\right)\sigma^2 + o(\sigma^2)\right)\int_{B_\sigma(x_0)}\abs{\nabla u}^2d\mu.
\end{align*}
\end{proof}

The following proposition comes from variation of the map $u$ on the target, rather than on the domain. In \cite{gromov-schoen} this is proved using the chain rule for maps into Riemannian simplicial complexes. In our more general setting, we do not have tools quite as strong as the chain rule, so we resort to the triangle comparison principle.

\begin{proposition}\label{target variation}
For a harmonic map $u:M\to X$ into an NPC space $X$ and for any $Q\in X$ we have in the weak sense

\[
\frac{1}{2}\Delta d^2(u,Q) \ge \abs{\nabla u}^2.
\]
\end{proposition}

\begin{proof}
Pick a point $x\in M$ and a tangent direction $v\in T_xM$ with $\abs{v}=1$. For $\epsilon>0$ small let $x+\epsilon v$ denote the image of $\epsilon v$ under the exponential map $exp_x$. For the points $Q$, $u(x)$ and $u(x+\epsilon v)$ in $X$, let the points $0$, $\tu(x)$ and $\tu(x+\epsilon v)$ in $\R^2$ be a comparison triangle. That is, the distances between the three points in $X$ coincides with the corresponding distances in $\R^2$.

Extend the map $\tu$ to the geodesic between $x$ and $x+\epsilon v$ by the formula
\[
\tu(x+s\epsilon v) = (1-s)\tu(x) + s\tu(x+\epsilon v).
\]
Take a smooth test function $\eta$ supported in a neighborhood of $x$ and define a one-parameter family of maps $u_t:M\to X$ defined by setting $u_t(y)$ to be the point that is a fraction $1-t\eta(y)$ of the way along the geodesic from $Q$ to $u(y)$. Similarly, along the geodesic from $x$ to $x+\epsilon v$ define

\[
\tu_t(y) = (1-t\eta(y))\tu(y).
\]

Now we can compute the directional derivative of the map $\tu_t$ using the chain rule:

\[
v\cdot\nabla \tu_t (y) = (1-t\eta(y))v\cdot\nabla\tu(y) - t(v\cdot\nabla\eta(y))\tu(y).
\]
And we can compute the magnitude of this derivative:
\begin{eqnarray*}
\abs{v\cdot\nabla\tu_t(y)}^2 & = & (1-t\eta(y))^2\abs{v\cdot\nabla\tu(y)}^2 - 2t(v\cdot\nabla\eta(y))\tu(y)\cdot(v\cdot\nabla\tu(y)) + O(t^2)\\
 & = & (1-2t\eta(y))\abs{v\cdot\nabla\tu(y)}^2 - t(v\cdot\eta(y))(v\cdot\nabla\abs{\tu}^2(y)) + O(t^2).
\end{eqnarray*}

A simple computation shows
\[
v\cdot\nabla{u}(x) = \lim_{s\to 0} \frac{ \tu(x+s\epsilon v) -\tu(x)}{s\epsilon} = \frac{\tu(x+\epsilon v)-\tu(x)}{\epsilon}.
\]
And hence
\[
\abs{v\cdot \nabla\tu(x)}^2 = \frac{d^2(u(x+\epsilon v),u(x))}{\epsilon^2}.
\]
Also,

\[
v\cdot\nabla \abs{\tu}^2(x) = \lim_{s\to 0}\frac{\abs{\tu(x+sv)}^2 - \abs{\tu(x)}^2}{s} = \frac{\abs{\tu(x+\epsilon v)}^2 - \abs{\tu(x)}^2}{\epsilon} + o(1).
\]

Now we have

\begin{eqnarray*}
\abs{v\cdot\tu_t(x)}^2 & = & (1-2t\eta(x))\frac{d^2(u(x+\epsilon v),u(x))}{\epsilon^2}\\
 & & - t(v\cdot\nabla\eta(x))\left( \frac{d^2(u(x+\epsilon v),Q) - d^2(u(x),Q)}{\epsilon} + o(1)\right) + O(t^2).
\end{eqnarray*}
By the triangle comparison principle, we have
\[
\frac{d(u_t(x),u_t(x+\epsilon v))}{\epsilon^2} \le \frac{\abs{\tu(x)-\tu(x+\epsilon v)}^2}{\epsilon^2} = \abs{v\cdot \nabla\tilde{u}(x)}^2 + o(1).
\]
Taking $\epsilon\to 0$ yields
\[
\abs{(u_t)_*v}^2(x) \le (1-2t\eta(x))\abs{u_*v}^2(x) - t(v\cdot\nabla\eta(x))(v\cdot\nabla d^2(u,Q)(x)) + O(t^2).
\]
Averaging over the sphere of unit vectors at $x$, we get
\[
\abs{\nabla u_t}^2(x) \le (1-2t\eta(x))\abs{\nabla u}^2(x) - t(\nabla \eta\cdot\nabla d^2(u,Q))(x) + O(t^2).
\]
Finally, integrate over the domain, and recall that $u$ is minimizing:
\[
E(u)\le E(u_t) \le E(u) - 2t\int \eta\abs{\nabla u}^2 + t\int (\Delta\eta)d^2(u,Q) + O(t^2).
\]
Cancelling the $E(u)$, dividing by $t$ and sending $t\to0$ yields

\[
\int (\Delta\eta)d^2(u,Q) \ge 2\int \eta\abs{\nabla u}^2.
\]
\end{proof}

In particular, letting $\eta$ approach the characteristic function of the ball $B_\sigma(0)$, we see

\[
\int_{B_\sigma}\abs{\nabla u}^2d\mu \le \int_{B_\sigma}\Delta d^2(u,Q) d\mu = \int_{\partial B_\sigma}\frac{\partial}{\partial r}d^2(u,Q)d\Sigma.
\]

We would like now to derive a stronger result if the target has a stronger curvature bound.

\begin{proposition}\label{negative cat target variation}
For a harmonic map $u:M\to X$ into a CAT(-1) space $X$ and for any $Q\in X$ we have

\[
\frac{1}{2}\Delta d^2(u,Q) \ge \abs{\nabla u}^2 + \left(\frac{d(u,Q)\cosh d(u,Q)}{\sinh d(u,Q)}-1\right)\left(\abs{\nabla u}^2 - \abs{\nabla d(u,Q)}^2\right).
\]
\end{proposition}

\begin{proof}
The strategy for this proof is identical to that of Proposition~\ref{target variation}. But now for the points $Q$, $u(x)$, and $u(x+\epsilon v)$, we have comparison points in $\Hyp^2$. We will choose normal coordinates in $\Hyp^2$, identifying it with $\R^2$ except with the metric $ds^2 = dr^2 + \sinh^2(r)d\theta^2$.

So take our comparison points to be $0$, $\tilde{u}(x)$, and $\tilde{u}(x+\epsilon v)$. Extend the map $\tilde{u}$ to map the geodesic joining $x$ and $x+\epsilon v$ to the geodesic joining $\tilde{u}(x)$ and $\tilde{u}(x+\epsilon v)$, and define $\tilde{u}_t$ on that geodesic segment by $\tilde{u}_t(y) = (1-t\eta(y))\tilde{u}(y)$ for some smooth function $\eta$. Just as before, the directional derivative of $\tilde{u}_t$ at $x$ is

\[
v\cdot\nabla\tilde{u}_t(x) = (1-t\eta(y))v\cdot\nabla\tilde{u}(x) - t(v\cdot\nabla\eta(x))\tilde{u}(x).
\]

But now this vector is based at $\tilde{u}_t(x)$, so we must use the metric at that point to compute its norm. The vector $\tu(x)$ is already pointing in the radial direction when it is based at $\tu(x)$ or even at $\tu_t(x)$. For $v\cdot\nabla\tu(x)$, the radial component is simply $v\cdot\nabla\abs{\tu}_0(x)$ and so the component orthogonal to the radial has length $\sqrt{\abs{v\cdot\nabla\tu(x)}^2_{\tu(x)} - \abs{v\cdot\nabla\abs{\tu}_0(x)}^2}$ at $\tu(x)$. Now in these coordinates the Euclidean length of the tangential component of $v\cdot\nabla\tu$ is $\frac{\abs{\tu(x)}_0}{\sinh\abs{\tu(x)}_0}$ times its length at $\tu(x)$, and so its length when placed at $\tu_t(x)$ is $\frac{\sinh\abs{\tu_t(x)}_0}{(1-t\eta(x))\sinh\abs{\tu(x)}_0}\sqrt{\abs{v\cdot\nabla\tu(x)}^2_{\tu(x)} - \abs{v\cdot\nabla\abs{\tu}_0(x)}^2}$.

Here I am using subscripts on norms of quantities on the target to indicate where those norms are being taken. In particular, $\abs{\tu(x)}_0$ is the distance from $\tu(x)$ to $0$. Now we can compute

\begin{eqnarray*}
\abs{v\cdot\nabla\tu_t(x)}^2_{\tu_t(x)} & = & \Big[ (1-t\eta(x))(v\cdot\nabla\abs{\tu}_0(x)) - t(v\cdot\nabla\eta(x))\abs{\tu(x)}_0\Big]^2\\
 & & + \frac{\sinh^2\abs{\tu_t(x)}_0}{\sinh^2\abs{\tu(x)}_0}\left(\abs{v\cdot\nabla\tu(x)}^2_{\tu(x)} - \abs{v\cdot\nabla\abs{\tu}_0(x)}^2\right).
\end{eqnarray*}

By Taylor expansion we have

\begin{eqnarray*}
\sinh(z) & = & \sum_{n=0}^\infty \frac{z^{2n+1}}{(2n+1)!}\\
\sinh((1-t\eta)z) & = & \sum_{n=0}^\infty \frac{(1-t\eta)^{2n+1}z^{2n+1}}{(2n+1)!}\\
 & = & \sum_{n=0}^\infty \frac{(1-(2n+1)t\eta)z^{2n+1}}{(2n+1)!} + O(t^2)\\
 & = & \sinh(z) - t\eta z\cosh(z) + O(t^2)\\
\sinh^2((1-t\eta)z) & = & \sinh^2(z) - 2t\eta z\sinh(z)\cosh(z) + O(t^2)\\
\frac{\sinh^2\abs{\tu_t(x)}_0}{\sinh^2\abs{\tu(x)}_0} & = & 1 - 2t\eta\frac{\abs{\tu(x)}_0\cosh\abs{\tu(x)}_0}{\sinh\abs{\tu(x)}_0} + O(t^2).
\end{eqnarray*}
So now we have
\begin{eqnarray*}
\abs{v\cdot\nabla\tu_t(x)}^2_{\tu_t(x)} & = & (1-2t\eta(x))\abs{v\cdot\nabla\abs{\tu}_0(x)}^2 - 2t\abs{\tu(x)}_0(v\cdot\nabla\abs{\tu}_0(x))(v\cdot\nabla\eta(x))\\
 & & + \left(1 - 2t\eta(x)\frac{\abs{\tu(x)}_0\cosh\abs{\tu(x)}_0}{\sinh\abs{\tu(x)}_0}\right)\left(\abs{v\cdot\nabla\tu(x)}^2_{\tu(x)} - \abs{v\cdot\nabla\abs{\tu}_0(x)}^2\right) + O(t^2)\\
 & = & (1-2t\eta(x))\abs{v\cdot\nabla\tu(x)}^2_{\tu(x)} - t(v\cdot\nabla\abs{\tu}^2_0(x))(v\cdot\nabla\eta(x)) + O(t^2)\\
 & & - 2t\eta(x)\left(\frac{\abs{\tu(x)}_0\cosh\abs{\tu(x)}_0}{\sinh\abs{\tu(x)}_0}-1\right)\left(\abs{v\cdot\nabla\tu(x)}^2_{\tu(x)} - \abs{v\cdot\nabla\abs{\tu}_0(x)}^2\right).
\end{eqnarray*}
%Now, since $\abs{v\cdot\nabla\tu(x)}_{\tu(x)} = \frac{d(\tu(x),\tu(x+\epsilon v))}{\epsilon}$, we have $\abs{v\cdot\nabla\tu(x)}_{\tu(x)} \to \abs{u_*v(x)}$ as $\epsilon\to 0$. And by the triangle comparison property, we have
The remainder of the argument follows exactly as in Proposition~\ref{target variation}.
\end{proof}

We have already used (in Lemma~\ref{monotonicity1}) and will continue to use extensively the fact that a harmonic map is locally Lipschitz continuous. This was shown in \cite{gromov-schoen} for maps into simplicial complexes and in \cite{korevaar-schoen1} for maps into NPC spaces. It also follows from our work if one takes only the weaker statements without assuming continuity, so we include the proof here for completeness.

\begin{proposition}
Let $u:M\to X$ be a harmonic map from a Riemannian manifold $M$ into and NPC metric space $X$. Then $u$ is locally Lipschitz continuous, with Lipschitz constant depending on the energy of $u$ and the injectivity radius of $M$.
\end{proposition}

\begin{proof}
First, we have equation~\eqref{monotonicity-gs} but not the result of Lemma~\ref{monotonicity1}. Following that, we will have only the following statement out of \cite{gromov-schoen}, which we expand near $0$, rather than the full force of Proposition~\ref{order}:

\[
\frac{\sigma \int_{B_\sigma}\abs{\nabla u}^2d\mu}{\int_{\partial B_\sigma}d^2(u,Q)d\Sigma} \ge e^{-c_1\sigma^2} = 1 + O(\sigma^2).
\]
And finally, in the proof of Proposition~\ref{mean value inequality} we would see
\[
\frac{d}{d\sigma} \log\left(\fint_{B_\sigma}\abs{\nabla u}^2 d\mu\right) \ge O(\sigma).
\]

Since the function $d^2(u,Q)$ is subharmonic for any $Q\in X$, the mean value inequality from \cite{li-schoen} says

\[
\sup_{B_{r/2}(x_0)}d^2(u,Q) \le C(1+O(r))\fint_{B_r(x_0)}d^2(u,Q)d\mu.
\]
Here the constant $C$ depends on the dimension of $M$, while the $O(r)$ term depends on the Ricci curvatures. If $M$ has non-negative Ricci curvatures, we may take $O(r) = 0$.

Now for $x\in M$, $r$ smaller than the injectivity radius at $x$, and any $Q\in X$ we have

\begin{eqnarray*}
\sup_{B_{r/2}(x)}d^2(u,Q) & \le & C(1+O(r))\fint_{B_r(x)}d^2(u,Q)d\mu\\
 & = & \frac{C + O(r)}{\abs{B_r(x)}}\int_0^r \int_{\partial B_s(x)}d^2(u,Q)d\Sigma ds\\
 & \le & \frac{C+O(r)}{\abs{B_r(x)}}\int_0^r (1+O(s))s\int_{B_s(x)}\abs{\nabla u}^2d\mu ds\\
 & \le & O(r^2)\fint_{B_r(x)}\abs{\nabla u}^2d\mu.
\end{eqnarray*} 

Fix $x_0\in M$ and $\sigma$ smaller than its injectivity radius. For $x_1,x_2\in B_{\sigma/8}(x_0)$ we have $d(x_1,x_2)<\sigma/4$. Take $Q$ to be the midpoint of $u(x_1)$ and $u(x_2)$. Note that

\[
d^2(u(x_1),u(x_2)) \le 2d^2(u(x_1),Q)+2d^2(u(x_2),Q).
\]

Applying the above inequality around $x_i$ for $i=1,2$ with radius $r = 2d(x_1,x_2)$, we have

\[
d^2(u(x_i),Q) \le O(r^2)\fint_{B_r(x_i)}\abs{\nabla u}^2 d\mu.
\]
Integrating the inequality from Proposition~\ref{mean value inequality} gives
\[
\fint_{B_{\sigma/2}(x_i)}\abs{\nabla u}^2d\mu \ge (1+O(\sigma^2))\fint_{B_r(x_i)}\abs{\nabla u}^2 d\mu.
\]
Hence
\begin{eqnarray*}
d^2(u(x_i),Q) & \le & O(r^2)(1+O(\sigma^2))\fint_{B_{\sigma/2}(x_i)}\abs{\nabla u}^2 d\mu\\
 & \le & \frac{O(r^2)(1+O(\sigma^2))}{\abs{B_{\sigma/2}(x_i)}}\int_{B_{\sigma/2}(x_0)}\abs{\nabla u}^2 d\mu.
\end{eqnarray*}

This $O(r^2)$ term we can say is less than some $Cr^2 = \frac{C}{4}d^2(x_1,x_2)$. Putting everything together yields, for some constant $C(\sigma)$ depending on $\sigma$,
\[
\frac{d^2(u(x_1),u(x_2))}{d^2(x_1,x_2)} \le C(\sigma)\int_{B_\sigma(x_0)}\abs{\nabla u}^2 d\mu.
\]
As this holds for any $x_1,x_2\in B_{\sigma/8}(x_0)$, this gives a local Lipschitz bound, and it can be seen how the bound depends on the energy of $u$ as well as the injectivity radius (i.e. distance to the boundary).
\end{proof}

\section{A Bochner Formula}\label{Sbochner}

This section of the paper will first give an asymptotic expansion of the usual monotonicity formula for harmonic maps before deriving the mean value property that will lead to our main theorem.

For most of this section, we will be working around a single fixed point $x_0\in M$, and computing in normal coordinates around $x_0$. So until we reach Proposition~\ref{mean value inequality}, fix $x_0$ to be a Lebesgue point for $pi$ and such that $\abs{\nabla u}^2(x_0)\neq0$. In a neighborhood of $x_0$, identify points $x\in M$ with vectors in $T_{x_0}M$ via the exponential map, in particular identifying $x_0$ with $0$. Also for simplicity of notation, we will use $B_\sigma$ to denote $B_\sigma(x_0) = B_\sigma(0)$.

Now define
\[
E(\sigma) = \int_{B_\sigma}\abs{\nabla u}^2d\mu
\]
And for a point $Q\in X$ define
\[
I(\sigma) = \int_{\partial B_\sigma}d^2(u,Q)d\mu
\]

We would like to choose the point $Q\in X$ that minimizes the integral $I(\sigma)$, at least asymptotically as $\sigma\to 0$. Since harmonic maps are locally Lipschitz continuous (c.f. \cite{gromov-schoen}, \cite{korevaar-schoen1}), we will always take $Q = u(x_0)$.

We will use frequently expansions of $E(\sigma)$ and $I(\sigma)$. For $E(\sigma)$, we will take advantage of $x_0$ being a Lebesgue point for $\abs{\nabla u}^2$ to say

\[
E(\sigma) = \abs{\nabla u}^2(0)Vol(B_\sigma) + o(\sigma^n) = \omega_n\abs{\nabla u}^2(0)\sigma^n + o(\sigma^n).
\]
For $I(\sigma)$, we first work only at a single point $x$ to derive a pointwise expression for $d^2(u,Q)$ before arriving at an expansion of $I(\sigma)$.

\begin{lemma}\label{d squared}
\[
d^2(u(x),u(0)) = \pi_{ij}(0)x^ix^j + e(x).
\]
Here $e(x)$ depends on the basepoint $0$ as well as $x$ near $0$, but $\frac{e(x)}{\abs{x}^2}$ can be bounded independent of both $x$ and $0$. Integrating $\frac{e(x)}{\abs{x}^2}$ along concentric spheres the quantity tends to $0$ at almost every basepoint:
\[
I(\sigma) = \omega_n\abs{\nabla u}^2(0)\sigma^{n+1} + o(\sigma^{n+1}).
\]
\end{lemma}

\begin{proof}
We recall the definition of the directional derivatives from \cite{korevaar-schoen1}, for almost every direction $Z$:

\[
\abs{u_*Z}^2(0) = \lim_{\epsilon\to 0} \frac{d^2(u(\epsilon Z),u(0))}{\epsilon^2}.
\]
We also recall that the pull back tensor $\pi$ satisfies
\[
\abs{u_*Z}^2(0) = \pi_{ij}(0)Z^iZ^j.
\]

%So we may write

%\[
%d^2(u(\epsilon Z),u(0)) = \pi_{ij}(\epsilon Z^i)(\epsilon Z^j) + o(\epsilon^2)
%\]

Hence the quantity
\[
\frac{d^2(u(\epsilon Z),u(0))}{\epsilon^2} - \pi_{ij}(0)Z^iZ^j
\]
tends to $0$ as $\epsilon\to 0$ for almost every direction $Z$. Say $\abs{Z}=1$, and letting $x = \epsilon Z$ we define

\[
e(x) = d^2(u(x),u(0)) - \pi_{ij}(0)x^ix^j.
\]

%\[
%d^2(u(x),u(0)) = \pi_{ij}x^ix^j + o(\abs{x}^2)
%\]

Since $u$ is locally Lipschitz, we can say $\frac{e(x)}{\abs{x}^2}$ is bounded independent of $\abs{x}$ and the basepoint $0$. Integrating over a sphere of radius $\sigma$ we see

\[
I(\sigma) = \omega_n\abs{\nabla u}^2(0)\sigma^{n+1} + \int_{\partial B_\sigma} e(x)d\Sigma.
\]

The quantity $\frac{1}{\sigma^{n+1}}\int_{\partial B_\sigma}e(x)d\Sigma$ measures the difference between the $\sigma$-approximate energy density at $0$ and the actual energy density. This difference tends towards $0$ at almost every point $0\in M$. Together with the boundedness property, this shows

\[
I(\sigma) = \omega_n\abs{\nabla u}^2(0)\sigma^{n+1} + o(\sigma^{n+1}).
\]
%This $o(\sigma^{n+1})$ does indeed tend to $0$ in $L^1$ after division by $\sigma^{n+1}$ since this is how the energy density was defined in \cite{korevaar-schoen1} and the map is locally Lipschitz.
\end{proof}

For CAT(-1) targets we will also need to understand $\nabla d(u,Q)$. Since we expect the same result whether $X$ is CAT(-1) or just NPC, and the computations are easier if we compare to Euclidean space, we will only really use the weaker assumption that $X$ is NPC.

\begin{lemma}\label{grad d squared}
\[
\int_{B_\sigma}d^2(u,Q)\abs{\nabla d(u,Q)}^2d\mu = \frac{\omega_n}{n+2}\pi:\pi(0)\sigma^{n+2} + o(\sigma^{n+2}).
\]
Here $\pi:\pi$ means $g^{ik}g^{j\ell}\pi_{ij}\pi_{k\ell}$, which at the center of normal coordinates becomes $\pi_{ij}\pi_{ij}$.
\end{lemma}

\begin{proof}
For a point $x\in B_\sigma$ and a unit vector $v\in T_{x_0}M$, consider the points $u(0)=Q$, $u(x)$, $u(x+\epsilon v)$, and $u((1-\epsilon)(x+\epsilon v))$. Construct a subembedding in $\R^2$, i.e. four points $0$, $\tu(x)$, $\tu(x+\epsilon v)$, and $\tu((1-\epsilon)(x+\epsilon v))$ such that $d(u(x+\epsilon v),Q)\le\abs{\tu(x+\epsilon v)}$, $d(u(1-\epsilon)(x+\epsilon v),u(x))\le\abs{\tu((1-\epsilon)(x+\epsilon v))-\tu(x)}$ and all other corresponding distances are equal.

Now extend $\tu$ to be an affine map that interpolates the points $\tu(x)$, $\tu(x+\epsilon v)$ and $\tu((1-\epsilon)(x+\epsilon v))$. It will be given by

\[
\tu(tx+s\epsilon v) = t\tu(x)+s(\tu(x+\epsilon v)-\tu(x)) + \frac{t-1}{\epsilon}((1-\epsilon)\tu(x+\epsilon v)-\tu((1-\epsilon)(x+\epsilon v)).
\]

Now we compute

\begin{eqnarray*}
v\cdot\nabla d^2(u,Q)(x) & = & \lim_{s\to0}\frac{d^2(u(x+s\epsilon v),Q)-d^2(u(x),Q)}{s\epsilon}\\
 & \le & \lim_{s\to0}\frac{\abs{\tu(x+s\epsilon v)}^2-\abs{\tu(x)}^2}{s\epsilon}\\
 & = & 2\tu(x)\cdot\frac{\tu(x+\epsilon v)-\tu(x)}{\epsilon}\\
 & = & 2(x\cdot\nabla\tu(x))\cdot(v\cdot\nabla\tu(x))\\
 & = & \frac{1}{1-\epsilon}\bigg(\abs{x\cdot\nabla\tu(x)}^2 + (1-\epsilon)^2\abs{v\cdot\nabla\tu(x)}^2\\
 & & - \abs{((1-\epsilon)v-x)\cdot\nabla\tu(x)}^2\bigg).
\end{eqnarray*}

The first term is simply $\abs{x\cdot\nabla\tu(x)}^2=\abs{\tu(x)}^2=d^2(u(x),Q) = \pi_{ij}(x)x^ix^j + e(x)$. The second term is $(1-\epsilon)^2\abs{v\cdot\nabla\tu(x)}^2 = (1-\epsilon)^2\left(\frac{\abs{\tu(x+\epsilon v)-\tu(x)}^2}{\epsilon^2}+o(1)\right) = \abs{u_*v}^2(x)+o(1)$. The last term is $\frac{\abs{\tu((1-\epsilon)(x+\epsilon v))-\tu(x)}^2}{\epsilon^2}+o(1) \ge \abs{u_*((1-\epsilon)v-x)}^2(x)+o(1) = \abs{u_*(v-x)}^2(x)+o(1)$. Putting this all together and sending $\epsilon\to0$, recall Lemma~\ref{d squared} and see

\begin{eqnarray*}
v\cdot\nabla d^2(u,Q)(x) & \le & d^2(u(x),Q) + \abs{u_*v}^2(x) - \abs{u_*(v-x)}^2(x)\\
 & = & 2\pi_{ij}(x)x^iv^j + e(x).
\end{eqnarray*}

Considering taking a partial derivative in the opposite direction, $-v$, we would have
\[
(-v)\cdot\nabla d^2(u,Q)(x) \le -2\pi_{ij}(x)x^iv^j + e(x).
\]
Hence we have an equality! The sign of the $e(x)$ term does not matter, as it is sufficiently small, so it will become a little-oh term upon integration. Now we use Proposition~\ref{quadratic form}

\begin{eqnarray*}
\abs{\nabla d^2(u,Q)}^2(x) & = & \frac{1}{\omega_n}\int_{\partial B_1}(v\cdot\nabla d^2(u,Q))^2(x)dS(v)\\
 & = & \frac{1}{\omega_n}\int_{\partial B_1}4\pi_{ij}(x)\pi_{k\ell}(x)x^iv^jx^kv^\ell dS(v) + e(x)\\
 & = & 4\pi_{ik}(x)\pi_{jk}(x)x^ix^j + e(x)\\
\int_{B_\sigma}\abs{\nabla d^2(u,Q)}^2d\mu & = & 4\int_{B_\sigma}\pi_{ik}(x)\pi_{jk}(x)x^ix^j(1+O(\sigma^2))dx+o(\sigma^{n+2})\\
 & = & 4\int_{B_\sigma}\pi_{ik}(0)\pi_{jk}(0)x^ix^jdx + o(\sigma^{n+2})\\
 & = & 4\frac{\omega_n}{n+2}\pi:\pi(0)\sigma^{n+2}+o(\sigma^{n+2}).
\end{eqnarray*}
\end{proof}

Now we will derive an inequality relating $E$, $I$, and $\abs{\frac{\partial u}{\partial r}}$:

\begin{lemma}\label{energy bound}
Let $u:M\to X$ be a harmonic map into a metric space $X$. If $X$ is NPC, then
\[
E(\sigma) \le \left(I(\sigma)\int_{\partial B_\sigma}\abs{\frac{\partial u}{\partial r}}^2d\Sigma\right)^\frac{1}{2}.
\]
If $X$ is CAT(-1), then
\[
E(\sigma) \le \left[1+\frac{\pi:\pi(0)-\abs{\nabla u}^4(0)}{3(n+2)\abs{\nabla u}^2(0)}\sigma^2+o(\sigma^2)\right]\left(I(\sigma)\int_{\partial B_\sigma}\abs{\frac{\partial u}{\partial r}}^2d\Sigma\right)^\frac{1}{2}.
\]
\end{lemma}

\begin{proof}
First the triangle inequality gives

\begin{equation}\label{radial derivatives}
\abs{\frac{\partial}{\partial r}d(u,Q)} \le \abs{\frac{du}{dr}}.
\end{equation}

If $X$ is NPC, apply the divergence theorem and the Cauchy-Schwarz inequality to Proposition~\ref{target variation}:

\begin{eqnarray*}
E(\sigma) & \le & \frac{1}{2}\int_{B_\sigma}\Delta d^2(u,Q)d\mu\\
 & = & \frac{1}{2}\int_{\partial B_\sigma}\frac{\partial}{\partial r}d^2(u,Q)d\Sigma\\
 & = & \int_{\partial B_\sigma}d(u,Q)\frac{\partial}{\partial r}d(u,Q)d\Sigma\\
 & \le & \left(I(\sigma)\int_{\partial B_\sigma}\abs{\frac{\partial u}{\partial r}}^2d\Sigma\right)^\frac{1}{2}.
\end{eqnarray*}

If $X$ is CAT(-1), apply the divergence theorem and Cauchy-Schwarz to Proposition~\ref{negative cat target variation}:

\begin{eqnarray*}
\frac{1}{2}\Delta d^2(u,Q) & = & d(u,Q)\Delta d(u,Q) + \abs{\nabla d(u,Q)}^2\\
 & \ge & \frac{d(u,Q)\cosh d(u,Q)}{\sinh d(u,Q)}\abs{\nabla u}^2 - \left(\frac{d(u,Q)\cosh d(u,Q)}{\sinh d(u,Q)} - 1\right)\abs{\nabla d(u,Q)}^2\\
\abs{\nabla u}^2 & \le & \tanh d(u,Q) \Delta d(u,Q) + \abs{\nabla d(u,Q)}^2\\
E(\sigma) & = & \int_{B_\sigma}\abs{\nabla u}^2\\
 & \le & \int_{B_\sigma}\tanh d(u,Q) \Delta d(u,Q) + \int_{B_\sigma}\abs{\nabla d(u,Q)}^2\\
 & = & \int_{\partial B_\sigma}\tanh d(u,Q)\frac{\partial}{\partial r}d(u,Q) + \int_{B_\sigma}\tanh^2 d(u,Q)\abs{\nabla d(u,Q)}^2\\
 & \le & \left(\int_{\partial B_\sigma}\tanh^2 d(u,Q)\right)^\frac{1}{2}\left(\int_{\partial B_\sigma}\abs{\frac{\partial u}{\partial r}}^2\right)^\frac{1}{2} + \int_{B_\sigma}d^2(u,Q)\abs{\nabla d(u,Q)}^2 + O(\sigma^{n+4}).
\end{eqnarray*}

Now, let's compare $\int_{\partial B_\sigma}\tanh^2 d(u,Q)$ with $I(\sigma)$. We have $\tanh^2(z) = z^2 - \frac{2}{3}z^4 + O(z^6)$, so
\begin{eqnarray*}
\frac{1}{I(\sigma)}\int_{\partial B_\sigma}\tanh^2 d(u,Q) & = & 1 - \frac{2}{3I(\sigma)}\int_{\partial B_\sigma}d^4(u,Q) + O(\sigma^{4})\\
 & = & 1 - \frac{2(2\pi:\pi(0) + \abs{\nabla u}^4(0))}{3(n+2)\abs{\nabla u}^2(0)}\sigma^2 + o(\sigma^2).
\end{eqnarray*}
From Lemma~\ref{grad d squared} we have
\[
\int_{B_\sigma}d^2(u,Q)\abs{\nabla d(u,Q)}^2 = \frac{\omega_n}{n+2}\pi:\pi(0)\sigma^{n+2}+o(\sigma^{n+2}) = \left[\frac{\pi:\pi(0)}{(n+2)\abs{\nabla u}^2(0)}\sigma^2 + o(\sigma^2)\right]E(\sigma).
\]
Moving this term to the other side and combining all the ingredients,

\begin{eqnarray*}
\Bigg[1 - \frac{\pi:\pi(0)}{(n+2)\abs{\nabla u}^2(0)}\sigma^2 + o(\sigma^2)\Bigg]E(\sigma) & \le & \Bigg[1 - \frac{2\pi:\pi(0)+\abs{\nabla u}^4(0)}{3(n+2)\abs{\nabla u}^2(0)}\sigma^2\\
 & & + o(\sigma^2)\Bigg]\left(I(\sigma)\int_{\partial B_\sigma}\abs{\frac{\partial u}{\partial r}}^2\right)^\frac{1}{2}.
\end{eqnarray*}

And hence

\[
E(\sigma) \le \left[ 1 + \frac{\pi:\pi(0)-\abs{\nabla u}^4(0)}{3(n+2)\abs{\nabla u}^2(0)}\sigma^2 + o(\sigma^2)\right]\left(I(\sigma)\int_{\partial B_\sigma}\abs{\frac{\partial u}{\partial r}}^2\right)^\frac{1}{2}.
\]
\end{proof}

It was shown in \cite{gromov-schoen} that for a suitable constant $c_1\ge 0$, the function
\[
e^{c_1\sigma^2}\frac{\sigma E(\sigma)}{I(\sigma)}
\]
is non-decreasing in $\sigma$. They used the monotnicity of this function to define the order of $u$ at $x_0$ as

\[
ord(x_0) = \lim_{\sigma\to 0} e^{c_1\sigma^2}\cdot\frac{\sigma E(\sigma)}{I(\sigma)}.
\]

At points where $\abs{\nabla u}^2\neq 0$, \cite{gromov-schoen} showed that $u$ has order $1$. This can also be seen via our expansions of $E$ and $I$ from the beginning of this section. We can now see how the ratio $\frac{\sigma E(\sigma)}{I(\sigma)}$ behaves asymptotically near $\sigma=0$:

% Suppose our $x_0$ is such a point:

%\begin{eqnarray*}
%E(\sigma) & = & \int_{B_\sigma}\abs{\nabla u}^2d\mu\\
% & = & \abs{\nabla u}^2(0)Vol(B_\sigma) + o(\sigma^n)\\
% & = & \omega_n\sigma^n\abs{\nabla u}^2(0) + o(\sigma^n)\\
%I(\sigma) & = & \int_{\partial B_\sigma} d^2(u(x),u(0))d\Sigma\\
% & = & \int_{\partial B_\sigma} \pi_{ij}x^ix^jdS + o(\sigma^{n+2})\\
% & = & \omega_n\sigma^{n+1}\abs{\nabla u}^2(0) + o(\sigma^{n+2})\\
%\frac{\sigma E(\sigma)}{I(\sigma)} & = & 1 + o (1)
%\end{eqnarray*}

\begin{proposition}\label{order}
Let $u:M\to X$ be a harmonic map into NPC space. Then
\[
\frac{\sigma E(\sigma)}{I(\sigma)} \ge 1 + \frac{2Ric:\pi(0)}{3(n+2)\abs{\nabla u}^2(0)}\sigma^2 + o(\sigma^2).
\]
If $X$ is CAT(-1), then
\[
\frac{\sigma E(\sigma)}{I(\sigma)} \ge 1 + \frac{2Ric:\pi(0)+\abs{\nabla u}^4(0)-\pi:\pi(0)}{3(n+2)\abs{\nabla u}^2(0)}\sigma^2 + o(\sigma^2).
\]
\end{proposition}

\begin{proof}
First note that

\begin{eqnarray*}
\lim_{\sigma\to 0}\frac{\sigma E(\sigma)}{I(\sigma)} & = & \lim_{\sigma\to 0}e^{-c_1\sigma^2}e^{c_1\sigma^2}\frac{\sigma E(\sigma)}{I(\sigma)}\\
 & = & \lim_{\sigma\to0}e^{-c_1\sigma^2}\lim_{\sigma\to 0}e^{c_1\sigma^2}\frac{\sigma E(\sigma)}{I(\sigma)}\\
 & = & 1\cdot ord(x_0).
\end{eqnarray*}

For an $L^\infty$ function $f$ on $M$, compute

\begin{eqnarray*}
\frac{d}{d\sigma}\int_{\partial B_\sigma}fd\Sigma & = & \int_{\partial B_\sigma}\frac{\partial}{\partial r}(f\sqrt{g})dS + \frac{n-1}{\sigma}\int_{B_\sigma}f\sqrt{g}dS\\
 & = & \int_{\partial B_\sigma}\frac{\partial f}{\partial r}d\Sigma -\frac{1}{3}\int_{\partial B_\sigma}f(x)\frac{Ric_0(x,x)}{\abs{x}}dS\\
 & &  + \frac{n-1}{\sigma}\int_{\partial B_\sigma}fd\Sigma + O(\sigma^2)\int_{\partial B_\sigma}fdS.
\end{eqnarray*}
Applying this to $f(x) = d^2(u(x),Q) = d^2(u(x),u(0))$, we compute a logarithmic derivative:
\begin{eqnarray*}
\frac{I'(\sigma)}{I(\sigma)} & = & \frac{1}{I(\sigma)}\int_{\partial B_\sigma}\frac{\partial}{\partial r}d^2(u,Q)d\Sigma +\frac{n-1}{\sigma}+O(\sigma^2)\\
 & & - \frac{1}{3\sigma I(\sigma)}\int_{\partial B_\sigma}d^2(u,Q)Ric_0(x,x)dS.
\end{eqnarray*}
Using Lemma~\ref{d squared} to rewrite $d^2(u,Q)$ and then part \ref{product quadratic on sphere} from Proposition~\ref{quadratic form}, compute
\begin{eqnarray*}
\int_{\partial B_\sigma} d^2(u,Q)Ric_0(x,x)dS & = & \int_{\partial B_\sigma} (\pi_{ij}(0)x^ix^j+e(x)) R_{k\ell}(0)x^kx^\ell dS\\
 & = & \frac{\omega_n}{n+2}\left(2Ric:\pi(0) + S(0)\abs{\nabla u}^2(0)\right)\sigma^{n+3} + o(\sigma^{n+3}).
\end{eqnarray*}
Combining, see
\begin{eqnarray*}
\frac{I'(\sigma)}{I(\sigma)} & = & \frac{1}{I(\sigma)} \int_{\partial B_\sigma}\frac{\partial}{\partial r}d^2(u,Q)d\Sigma + \frac{n-1}{\sigma}\\
 & & - \frac{1}{3(n+2)}\left( \frac{2Ric:\pi(0)}{\abs{\nabla u}^2(0)} + S(0)\right)\sigma + o(\sigma).
\end{eqnarray*}

Now compute the logarithmic derivative of $E(\sigma)$:

\begin{eqnarray*}
\frac{E'(\sigma)}{E(\sigma)} & = & (E(\sigma))^{-1}\int_{\partial B_\sigma}\abs{\nabla u}^2d\Sigma\\
 & = & \frac{n-2}{\sigma} + \frac{2}{E(\sigma)}\int_{\partial B_\sigma}\abs{\frac{\partial u}{\partial r}}^2d\Sigma\\
 & & + \frac{1}{3(n+2)}\left(\frac{2Ric:\pi(0)}{\abs{\nabla u}^2(0)} - S(0)\right)\sigma + o(\sigma).
\end{eqnarray*}
This last line comes from Lemma~\ref{monotonicity1}. Now combine these:
\begin{eqnarray*}
\frac{d}{d\sigma}\log\left(\frac{\sigma E(\sigma)}{I(\sigma)}\right) & = & \frac{1}{\sigma} + \frac{E'(\sigma)}{E(\sigma)} - \frac{I'(\sigma)}{I(\sigma)}\\
 & = & - (I(\sigma))^{-1}\int_{\partial B_\sigma}\frac{\partial}{\partial r}(d^2(u,Q))d\Sigma\\
 & & + 2(E(\sigma))^{-1}\int_{\partial B_\sigma}\abs{\frac{\partial u}{\partial r}}^2d\Sigma\\
 & & + \frac{1}{3(n+2)}\cdot\frac{4Ric:\pi(0)}{\abs{\nabla u}^2(0)}\sigma + o(\sigma).
\end{eqnarray*}

From Lemma~\ref{energy bound} we have, for some constant $C$ depending on the curvature bound for $X$,

\begin{eqnarray*}
E(\sigma) & \le & \left(1+C\sigma^2+o(\sigma^2)\right)\left(I(\sigma)\int_{\partial B_\sigma}\abs{\frac{\partial u}{\partial r}}^2d\Sigma\right)^\frac{1}{2}\\
\frac{2}{E(\sigma)}\int_{\partial B_\sigma}\abs{\frac{\partial u}{\partial r}}^2d\Sigma & \ge & 2\left(1-C\sigma^2+o(\sigma^2)\right)\left(\frac{1}{I(\sigma)}\int_{\partial B_\sigma}\abs{\frac{\partial u}{\partial r}}^2d\Sigma\right)^\frac{1}{2}.
\end{eqnarray*}
And from Cauchy-Schwarz, we get
\[
\frac{1}{I(\sigma)}\int_{\partial B_\sigma}\frac{\partial}{\partial r}(d^2(u,Q))d\Sigma \le 2\left(\frac{1}{I(\sigma)}\int_{\partial B_\sigma}\abs{\frac{\partial u}{\partial r}}^2d\Sigma\right)^\frac{1}{2}.
\]
Note also that $\frac{1}{I(\sigma)}\int_{\partial B_\sigma}\abs{\frac{\partial u}{\partial r}}^2d\Sigma = \sigma^{-2}+o(\sigma^{-2})$, and so conclude
\[
\frac{d}{d\sigma}\log\left(\frac{\sigma E(\sigma)}{I(\sigma)}\right) \ge \left(\frac{4Ric:\pi(0)}{3(n+2)\abs{\nabla u}^2(0)}-2C\right)\sigma + o(\sigma).
\]

Integrating over $\sigma$, recall Lemma~\ref{integrate-little o} and see

\[
\log\left(\frac{\sigma E(\sigma)}{I(\sigma)}\right) - \log(ord(0)) \ge \left(\frac{2Ric:\pi(0)}{3(n+2)\abs{\nabla u}^2(0)}-C\right)\sigma^2 + o(\sigma^2).
\]
Recall $ord(x_0)=1$ and $e^t=1+t+O(t^2)$. Taking exponentials now shows
\[
\frac{\sigma E(\sigma)}{I(\sigma)} \ge 1 + \left(\frac{2Ric:\pi(0)}{3(n+2)\abs{\nabla u}^2(0)}-C\right)\sigma^2 + o(\sigma^2).
\]
The value of $C$ for NPC or CAT(-1) targets that comes from Lemma~\ref{energy bound} yields the desired results.

\end{proof}

As a consequence of the order growth we have the following useful lemma:

\begin{lemma}\label{A vs I}
Let $u:M\to X$ be a harmonic map. If $X$ is NPC, we have
\[
\sigma\int_{\partial B_\sigma}\abs{\frac{\partial u}{\partial r}}^2d\Sigma \ge \left(1+\frac{2Ric:\pi(0)}{3(n+2)\abs{\nabla u}^2(0)}\sigma^2+o(\sigma^2)\right)E(\sigma).
\]
If $X$ is CAT(-1) then we have
\[
\sigma\int_{\partial B_\sigma}\abs{\frac{\partial u}{\partial r}}^2d\Sigma \ge \left(1+\frac{2Ric:\pi(0)+3\abs{\nabla u}^4(0)-3\pi:\pi(0)}{3(n+2)\abs{\nabla u}^2(0)}\sigma^2+o(\sigma^2)\right)E(\sigma).
\]
\end{lemma}

\begin{proof}
There are constants $C$ (coming from Lemma~\ref{energy bound}) and $C'$ (coming from Proposition~\ref{order}) that depend on the curvature bound of $X$ such that

\begin{eqnarray*}
E(\sigma) & \le & (1+C\sigma^2+o(\sigma^2))\left(I(\sigma)\int_{\partial B_\sigma}\abs{\frac{\partial u}{\partial r}}^2d\Sigma\right)^\frac{1}{2}\\
\frac{\sigma E(\sigma)}{I(\sigma)} & \ge & 1 + C'\sigma^2 + o(\sigma^2).
\end{eqnarray*}

Combining these, see
\begin{eqnarray*}
\sigma\left(I(\sigma)\int_{\partial B_\sigma}\abs{\frac{\partial u}{\partial r}}^2d\Sigma\right)^\frac{1}{2} & \ge & (1-C\sigma^2+o(\sigma^2))\sigma E(\sigma)\\
 & \ge & (1 + (C'-C)\sigma^2 + o(\sigma^2))I(\sigma)\\
\sigma\int_{\partial B_\sigma}\abs{\frac{\partial u}{\partial r}}^2d\Sigma & \ge & (1+(C'-C)\sigma^2+o(\sigma^2))\left(I(\sigma)\int_{\partial B_\sigma}\abs{\frac{\partial u}{\partial r}}^2d\Sigma\right)^\frac{1}{2}\\
 & \ge & (1+(C'-2C)\sigma^2+o(\sigma^2))E(\sigma).
\end{eqnarray*}
Supplying the appropriate constants yields the desired results.
\end{proof}

Now we are prepared to prove the mean value inequality that will lead to our main result.

\begin{proposition}\label{mean value inequality}
Let $u:M\to X$ be a harmonic map, and let $x_0\in M$ be almost any point. If $X$ is NPC then
\[
\fint_{B_\sigma(x_0)}\abs{\nabla u}^2d\mu \ge \abs{\nabla u}^2(x_0) + \frac{Ric:\pi(x_0)}{n+2}\sigma^2 + o(\sigma^2).
\]
If $X$ is CAT(-1) then
\[
\fint_{B_\sigma(x_0)}\abs{\nabla u}^2d\mu \ge \abs{\nabla u}^2(x_0) + \frac{Ric:\pi(x_0)+\abs{\nabla u}^4(x_0)-\pi:\pi(x_0)}{n+2}\sigma^2 + o(\sigma^2).
\]
\end{proposition}

\begin{proof}
If $\abs{\nabla u}^2(x_0) = 0$, then the inequality holds trivially. Let $x_0\in M$ with $\abs{\nabla u}^2(x_0)\neq 0$ be a point where the previous propositions and lemmas apply and identify $x_0$ with $0\in T_{x_0}M$. Compute, setting for simplicity $V_\sigma = Vol(B_\sigma(x_0))$:

\begin{eqnarray*}
\frac{d}{d\sigma}\fint_{B_\sigma}\abs{\nabla u}^2d\mu & = & \frac{1}{V_\sigma^2}\left( V_\sigma\int_{\partial B_\sigma}\abs{\nabla u}^2d\Sigma - V'_\sigma E(\sigma)\right)\\
 & = & \frac{1}{V_\sigma} \Big[ \int_{\partial B_\sigma}\abs{\nabla u}^2 d\Sigma - \frac{n}{\sigma} E(\sigma)\\
 & & + \left( \frac{1}{3(n+2)}S(0)\sigma + O(\sigma^2)\right) E(\sigma) \Big]\\
 & = & \frac{1}{V_\sigma} \Bigg[ 2\int_{\partial B_\sigma} \abs{\frac{\partial u}{\partial r}}^2 d\Sigma - \frac{2}{\sigma} E(\sigma)\\
 & & + \left( \frac{1}{3(n+2)}\cdot\frac{2Ric:\pi(0)}{\abs{\nabla u}^2(0)}\sigma + o(\sigma)\right)E(\sigma)\Bigg]\\
 & = & \frac{2}{\sigma V_\sigma}\Bigg[ \sigma\int_{\partial B_\sigma}\abs{\frac{\partial u}{\partial r}}^2 d\Sigma - E(\sigma)\\
 & & + \left( \frac{Ric:\pi(0)}{3(n+2)\abs{\nabla u}^2(0)}\sigma^2 + o(\sigma^2)\right)E(\sigma)\Bigg].
\end{eqnarray*}
The comparison of $V_\sigma$ and $V'_\sigma$ comes from Proposition~\ref{bishop-gromov}. Recall Lemma~\ref{A vs I} with a constant $A$ depending on the curvature of $X$, and apply it as follows:

\begin{eqnarray*}
\sigma\int_{\partial B_\sigma}\abs{\frac{\partial u}{\partial r}}^2d\Sigma & \ge & (1+A\sigma^2+o(\sigma^2))E(\sigma)\\
\frac{d}{d\sigma}\fint_{B_\sigma}\abs{\nabla u}^2d\mu & \ge & \frac{2}{\sigma V_\sigma}\left(\left(\frac{Ric:\pi(0)}{3(n+2)\abs{\nabla u}^2(0)}+A\right)\sigma^2+o(\sigma^2)\right)E(\sigma)\\
\frac{d}{d\sigma}\log\fint_{B_\sigma}\abs{\nabla u}^2d\mu & \ge & \left(\frac{2Ric:\pi(0)}{3(n+2)\abs{\nabla u}^2(0)}+2A\right)\sigma + o(\sigma).
\end{eqnarray*}

Integrating with respect to $\sigma$, recall that we are working at a Lebesgue point for $\abs{\nabla u}^2$ and see
\begin{eqnarray*}
\fint_{B_\sigma}\abs{\nabla u}^2d\mu & \ge & \abs{\nabla u}^2(0) \exp \left( \left(\frac{Ric:\pi(0)}{(n+2)\abs{\nabla u}^2(0)}+A\right)\sigma^2 + o(\sigma^2)\right)\\
 & = & \abs{\nabla u}^2(0)\left(1 + \left(\frac{Ric:\pi(0)}{(n+2)\abs{\nabla u}^2(0)}+A\right)\sigma^2 + o(\sigma^2)\right).
\end{eqnarray*}
Substituting the appropriate value for $A$ from Lemma~\ref{A vs I} yields the desired result.
\end{proof}

And now we'd like to turn the mean value inequality into a differential inequality.

\begin{proposition}\label{subharmonicity}
Let $f$ be an $L^\infty$ function on a relatively compact Riemannian domain $\Omega$ of dimension $n$ satisfying for almost every $x_0\in \Omega$
\[
\fint_{B_\sigma(x_0)}fd\mu \ge f(x_0) + \varphi(x_0)\sigma^2 + o(\sigma^2).
\]
Then $f$ satisfies the weak differential inequality
\[
\frac{1}{2}\Delta f \ge (n+2)\varphi.
\]
\end{proposition}

\begin{proof}
First we will need a computation about smooth functions. For a smooth function $h$ on $\Omega$, take normal coordinates about $x_0\in \Omega$, and the Taylor expansion of $h$ about $x_0$ in these coordinates is

\[
h(x) = h(0) + \partial_ih(0)x^i + \frac{1}{2}\partial_i\partial_jh(0)x^ix^j + O(\abs{x}^2).
\]

Now compute the average of $h$ on $B_\sigma(x_0)$, recalling that the volume density in normal coordinates is $\sqrt{g}=1+O(\abs{x}^2)$ and that the Riemannian volume of $B_\sigma$ is $\omega_n\sigma^n + O(\sigma^{n+2})$:
\begin{eqnarray*}
\fint_{B_\sigma}h d\mu & = & h(0) + \partial_i h(0)\fint_{B_\sigma}x^id\mu + \frac{1}{2}\partial_i\partial_jh(0)\fint_{B_\sigma}x^ix^jd\mu + O(\sigma^3)\\
 & = & h(0) + \partial_i h(0)\fint_{B_\sigma}x^idx + \frac{1}{2}\partial_i\partial_j h(0)\fint_{B_\sigma}x^ix^jdx + O(\sigma^3).
\end{eqnarray*}
Integrating now with respect to a Euclidean metric, $\fint_{B_\sigma}x^idx = 0$, and Proposition~\ref{quadratic form} implies
\[
\fint_{B_\sigma}\frac{1}{2}\partial_i\partial_jh(0)x^ix^jdx = \frac{1}{2(n+2)}\sigma^2\partial_i\partial_ih(0).
\]
And in normal coordinates, $\partial_i\partial_i h(0) = \Delta h(0)$, so
\[
\fint_{B_\sigma(x_0)}h d\mu = h(x_0) + \frac{1}{2(n+2)}\sigma^2\Delta h(x_0) + O(\sigma^3).
\]

This $O(\sigma^3)$ term depends on the geometry of $\Omega$ as well as higher order information about the function $h$. Note that this formula, when combined with the hypothesis, immediately yields the result for the smooth function $h$.

For $f\in L^1$, we will integrate against a smooth test function $\eta$. It will be useful to be able to move an average value operator from one function onto the other. In order to do so, we must first compare volumes of balls centered at different points:

\begin{eqnarray*}
\abs{B_\sigma(p)} & = & \omega_n\sigma^n\left(1 - \frac{S(p)}{6(n+2)}\sigma^2 + O(\sigma^3)\right)\\
\frac{1}{\abs{B_\sigma(x)}} - \frac{1}{\abs{B_\sigma(y)}} & = & \frac{1}{\omega_n\sigma^n}\left( \frac{S(x) - S(y)}{6(n+2)}\sigma^2 + O(\sigma^3)\right)\\
 & = & O(d(x,y)\sigma^{2-n}).
\end{eqnarray*}

For $\sigma < \frac{1}{2}d(supp(\eta),\partial\Omega)$ we can make sense of the following:

\begin{eqnarray*}
\int \eta(x)\left(\fint_{B_\sigma(x)}f(y)d\mu(y)\right)d\mu(x) & = & \int_{d(x,y)<\sigma}\frac{1}{\abs{B_\sigma(x)}}\eta(x)f(y)d\mu(y)d\mu(x)\\
 & = & \int f(y)\left(\frac{1}{\abs{B_\sigma(y)}}+O(\sigma^{3-n})\right)\int_{B_\sigma(y)}\eta(x)d\mu(x)d\mu(y)\\
 & = & \int f(y)\left(\fint_{B_\sigma(y)}\eta(x)d\mu(x) + O(\sigma^3)\eta(y)\right)d\mu(y).
\end{eqnarray*}

Now since $\eta$ is smooth it satisfies

\[
\fint_{B_\sigma(x)}\eta d\mu = \eta(x) + \frac{1}{2(n+2)}\Delta\eta(x)\sigma^2 + O(\sigma^3).
\]

Now we are in a position to integrate $\Delta\eta$ against $f$:
\begin{eqnarray*}
\int f\Delta\eta d\mu & = & \frac{2(n+2)}{\sigma^2}\int f(x)\left(\fint_{B_\sigma(x)}\eta d\mu - \eta(x) + O(\sigma^3)\right)d\mu(x)\\
 & = & \frac{2(n+2)}{\sigma^2}\int \eta(y)\left( \fint_{B_\sigma(y)}f d\mu + O(\sigma^3)f(y) - f(y)\right)d\mu(y)  + O(\sigma)\\
 & \ge & 2(n+2)\int\eta(y)(\varphi(y)+o(1))d\mu(y)+O(\sigma).
\end{eqnarray*}
Taking $\sigma\to 0$ yields the desired weak inequality
\[
\int f\Delta\eta d\mu \ge \int 2(n+2)\varphi\eta d\mu.
\]
\end{proof}

%\begin{theorem}\label{bochner}
%Let $u:M\to X$ be a harmonic map into an NPC space $X$. Then $u$ satisfies, in the weak sense, the inequality

%\[
%\frac{1}{2}\Delta\abs{\nabla u}^2 \ge tr(\pi\cdot Ric)
%\]

%If $X$ is CAT(-1), then $u$ satisfies the stronger inequality

%\[
%\frac{1}{2}\Delta\abs{\nabla u}^2 \ge tr(\pi\cdot Ric) + \abs{\nabla u}^4-tr(\pi^2)
%\]
%\end{theorem}

\begin{proof1}
From Proposition~\ref{mean value inequality} we have
\[
\fint_{B_\sigma(x)}\abs{\nabla u}^2d\mu \ge \varphi(x).
\]
Here the function $\varphi$ depends on the Ricci curvatures of $M$ and the curvature bound for $X$. Applying Proposition~\ref{subharmonicity}, we get our result.
\end{proof1}

\begin{remark}
Compare to the Bochner formula of \cite{eells-sampson}:

\[
\frac{1}{2}\Delta\abs{\nabla u}^2 = \abs{\nabla du}^2 + \la Ric^M(\nabla u),\nabla u\ra - \la R^X(\nabla u,\nabla u)\nabla u,\nabla u\ra.
\]

The terms $Ric:\pi$ that have appeared are exactly analogous to the $\la Ric(\nabla u),\nabla u\ra$ that appears in the classical formula. In addition, when the target is a manifold of curvature $-1$, the sectional curvature term $\langle R^X(\nabla u,\nabla u)\nabla u,\nabla u\rangle$ simplifies to $g^{ij}g^{k\ell}(\pi_{ij}\pi_{k\ell}-\pi_{ik}\pi_{j\ell}) = \abs{\nabla u}^4-\pi:\pi$.
\end{remark}

\begin{corollary}
Let $u:M\to X$ be a harmonic map from a compact manifold $M$ of non-negative Ricci curvatures into an NPC space $X$. Then $\abs{\nabla u}^2\equiv 0$. If the Ricci curvatures of $M$ are positive somewhere, then $u$ is a constant map. If $X$ is CAT(-1) and $u$ is not a constant map, then the rank of $\pi$ is almost everywhere equal to one.
\end{corollary}

%\begin{proof1}
\begin{proof}
Theorem~\ref{main1} implies that over $M$ we have
\[
\Delta\abs{\nabla u}^2 \ge 0.
\]
Since $M$ is compact,
\[
\int_M \Delta\abs{\nabla u}^2 d\mu = 0.
\]
Hence we must have $\Delta\abs{\nabla u}^2 = 0$ almost everywhere, so $\abs{\nabla u}^2$ is constant.

If in addition the Ricci curvatures of $M$ are positive at a single point of $M$, there is an open set $\Omega$ where the Ricci curvatures are strictly positive. Theorem~\ref{main1} now says that over $\Omega$ there is $m>0$ so that
\[
\frac{1}{2}\Delta\abs{\nabla u}^2 \ge m\abs{\nabla u}^2.
\]
Since $\Delta\abs{\nabla u}^2 = 0$, we must have $\abs{\nabla u}^2\equiv 0$ on $\Omega$. Since $\abs{\nabla u}^2$ is constant, it must be $0$ on all of $M$. That is, $u$ is a constant map.

If $X$ is CAT(-1) then the vanishing of $\abs{\nabla u}^4-\pi:\pi$ implies that $\pi$ has at most one positive eigenvalue. Since $\pi$ is positive semi-definite, if it is not zero then it has rank one.
\end{proof}

\begin{proof2}
Let $M$ be a closed hyperbolic manifold of dimension $n$. A map $u:M\to X$ being conformal means that the pull-back tensor satisfies $\pi = \lambda g$ for a non-negative function $\lambda$. If $u$ is harmonic and $X$ is CAT(-1) then Theorem~\ref{main1} says

\[
\frac{n}{2}\Delta\lambda \ge -n\lambda + (n\lambda)^2 - n\lambda^2 = n\lambda\left((n-1)\lambda-1\right).
\]

At a point where $\lambda$ is maximized, $\Delta\lambda\le 0$. At such a point,

\[
\lambda\le\frac{1}{n-1}.
\]
\end{proof2}

In the case that $X$ is CAT(-1) and so $\pi$ has rank one, we expect that the map $u$ should map into a geodesic on $X$. We will approach this problem in a later paper, as it will require several new tools. In the meantime, for a harmonic map from a flat torus into an NPC space, we can analyze all the inequalities we used in order to show the map is totally geodesic of harmonic maps. Since the energy density is constant, if it is not identically $0$ then every point has order $1$. The following arguments (essentially those presented in \cite{gromov-schoen} Lemma $3.2$) show that such a map must be totally geodesic:

\begin{proof3}
If $M$ is compact and flat, then in normal coordinates around any point, the metric is constant. Thus the formula \eqref{monotonicity-gs} is simply

\[
(2-n)\int_{B_\sigma(x)}\abs{\nabla u}^2 dx + \sigma\int_{B_\sigma(x)}\left[\abs{\nabla u}^2 - 2\abs{\frac{\partial u}{\partial r}}^2\right]dS = 0.
\]

The further lemmas and propositions then follow with no error terms. In particular, since $\abs{\nabla u}^2$ is constant by Theorem~\ref{main1}, we have
\[
0 = \frac{d}{d\sigma}\fint_{B_\sigma}\abs{\nabla u}^2d\mu \ge 0.
\]
Hence all of the inequalities used must be equalities. The Cauchy-Schwarz inequality gives a constant (depending on $\sigma$) so that on $\partial B_\sigma$ we have

\[
\frac{\partial}{\partial r}d(u,Q) = c_\sigma d(u,Q).
\]

Equation \eqref{radial derivatives} gives almost everywhere
\[
\abs{\frac{\partial}{\partial r}d(u,Q)} = \abs{\frac{\partial u}{\partial r}}.
\]
And from $\int_{\partial B_\sigma}d^2(u,Q)dS = \sigma^2\int_{\partial B_\sigma}\abs{\frac{\partial u}{\partial r}}^2dS$ we have $c_\sigma = \frac{1}{\sigma}$. Now we see

\[
\abs{\frac{\partial u}{\partial r}} = \frac{\partial}{\partial r}d(u,Q) = \frac{1}{\sigma}d(u,Q).
\]

First note that for $x$ near $0$ we calculate
\[
d(u(x),u(0)) = \int_0^x \frac{\partial}{\partial r}d(u,Q)dt = \int_0^x\abs{\frac{\partial u}{\partial r}}dt.
\]
And this last integral calculates the length of the image of the geodesic from $0$ to $x$. Hence the image of this geodesic segment is a geodesic in $X$. Now from
\[
\frac{\partial}{\partial r}d(u(x),u(0)) = \frac{1}{\abs{x}}d(u(x),u(0))
\]
conclude that the image of the unit speed geodesic from $0$ to $x$ is a constant speed geodesic in $X$. Since harmonic maps into NPC spaces are locally Lipschitz continuous (c.f.~\cite{gromov-schoen},\cite{korevaar-schoen1}), the speed of the image geodesics in $X$ depends only on the direction of the geodesics in $M$. Hence $u$ is totally geodesic.
\end{proof3}

Classically, the Bochner formula proves that any map from a space of non-negative Ricci curvature to a space of non-positive curvature is totally geodesic. The methods of this paper are not quite strong enough to detect this when the domain is not flat. The sectional curvatures of the domain influence many of the $o(\sigma^k)$ error terms. In order to show the map is totally geodesic, we need all of our inequalities to be equalities, but these error terms interfere. We nonetheless phrase this as a conjecture.

\begin{conjecture}
Any harmonic map from a Riemannian manifold of non-negative Ricci curvatures to an NPC metric space is totally geodesic.
\end{conjecture}

\end{document}